\documentclass[12pt,a4paper]{article}
\usepackage[utf8x]{inputenc}
\usepackage{ucs}
\usepackage{amsfonts, amssymb, amsmath, amsthm}
\usepackage{color}
\usepackage{graphicx}
\usepackage[lf]{Baskervaldx} 
\usepackage[bigdelims,vvarbb]{newtxmath} 
\usepackage[cal=boondoxo]{mathalfa} 

\usepackage[width=16.00cm, height=24.00cm, left=2.50cm]{geometry}

\newtheorem{theorem}{Theorem}
\newtheorem{lemma}[theorem]{Lemma}
\newtheorem{proposition}[theorem]{Proposition}
\newtheorem{corollary}[theorem]{Corollary}
\theoremstyle{definition}
\newtheorem{remark}[theorem]{Remark}

\usepackage[colorlinks=true,linkcolor=colorref,citecolor=colorcita,urlcolor=colorweb]{hyperref}
\definecolor{colorcita}{RGB}{21,86,130}
\definecolor{colorref}{RGB}{5,10,177}
\definecolor{colorweb}{RGB}{177,6,38}

\usepackage[shortlabels]{enumitem}

\DeclareMathOperator{\re}{Re}
\DeclareMathOperator{\essinf}{essinf}
\DeclareMathOperator{\ess}{ess}
\DeclareMathOperator{\gpd}{gpd}

\allowdisplaybreaks

\title{Multipliers for Hardy spaces of Dirichlet series}

\author{Tomás Fernández Vidal\thanks{Supported by CONICET-PIP 11220200102336} \and Daniel Galicer\thanks{Supported by PICT 2018-4250.}  \and Pablo Sevilla-Peris\thanks{Supported by MINECO and FEDER Project MTM2017-83262-C2-1-P and by GV Project AICO/2021/170}}

\date{}

\begin{document}
	
\maketitle	

\begin{abstract}
We characterize the space of multipliers from the Hardy space of Dirichlet series $\mathcal H_p$ into $\mathcal H_q$ for every $1 \leq p,q \leq \infty$. 
For a fixed Dirichlet series, we also investigate some structural properties of its associated multiplication operator. In particular, we study the norm, the essential norm, and the spectrum for an operator of this kind.
We exploit the existing natural identification of spaces of Dirichlet series with spaces of holomorphic functions in infinitely many variables and apply several methods from complex and harmonic analysis to obtain our results. As a byproduct we get analogous statements on such Hardy spaces of holomorphic functions.

\end{abstract}


\section{Introduction}

A Dirichlet series is a formal expression of the type $D=\sum a_n n^{-s}$ with $(a_n)$ complex values and $s$ a complex variable. 
These are one of the basic tools of analytic number theory (see e.g., \cite{apostol1984introduccion, tenenbaum_1995}) but, over the last two decades, as a result of the work initiated in \cite{hedenmalm1997hilbert} and \cite{konyaginqueffelec_2002}, they have been analyzed with techniques coming from harmonic and functional analysis (see e.g. \cite{queffelec2013diophantine} or \cite{defant2018Dirichlet} and the references therein). One of the key point in this analytic insight on Dirichlet series is the deep connection with power series in infinitely many variables. We will use this fruitful perspective to study multipliers for Hardy spaces of Dirichlet series. We begin by recalling some standard definitions of these spaces.

The natural regions of convergence of Dirichlet series are half-planes, and there they define holomorphic functions.
To settle some notation, we consider the set $\mathbb{C}_{\sigma} = \{ s \in \mathbb{C} \colon \re s > \sigma\}$, for $\sigma \in \mathbb{R}$.
With this, Queff\'elec \cite{Quefflec95} defined the space  $\mathcal{H}_{\infty}$ as that consisting of Dirichlet series that  define a bounded, holomorphic function on the half-plane $\mathbb{C}_{0}$. Endowed with the norm $\Vert D \Vert_{\mathcal{H}_\infty} := \sup\limits_{s\in \mathbb{C}_0} \vert \sum \frac{a_n}{n^s} \vert < \infty$ it becomes a Banach space, which together with the product $(\sum a_n n^{-s})\cdot (\sum b_n b^{-s}) = \sum\limits_{n =1}^{\infty} \big(\sum\limits_{k\cdot j = n} a_k\cdot b_j \big) n^{-s}$ results a Banach algebra.

The Hardy spaces of Dirichlet series $\mathcal{H}_p$ were introduced by Hedenmalm, Lindqvist and Seip \cite{hedenmalm1997hilbert} for $p=2$, and by Bayart \cite{bayart2002hardy} for the remaining cases in the range $1\leq p < \infty$. A way to define these spaces is to consider first the following norm in the space of  Dirichlet polynomials (i.e., all finite sums of the form $\sum_{n=1}^{N} a_{n} n^{-s}$, with $N \in \mathbb{N}$),

\[
\Big\Vert \sum_{n=1}^{N} a_{n} n^{-s} \Big\Vert_{\mathcal{H}_p}
	:= \lim_{R \to \infty} \bigg( \frac{1}{2R} \int_{-R}^{R} \Big\vert \sum_{n=1}^{N} a_{n} n^{-it} \Big\vert^{p}  dt  \bigg)^{\frac{1}{p}} \,,
\]
and define $\mathcal{H}_p$ as the completion of the Dirichlet polynomials under this norm. Each Dirichlet series in some $\mathcal{H}_{p}$ (with $1 \leq p < \infty$) converges on $\mathbb{C}_{1/2}$, and there it defines a holomorphic function.

The Hardy space $\mathcal H_p$ with the function product is not an algebra for $p<\infty$. Namely, given two Dirichlet series $D, E \in \mathcal{H}_p$, it is not true, in general, that the product function $D\cdot E$ belongs to $\mathcal{H}_p$. Nevertheless, there are certain series $D$ that verify that $D \cdot E \in \mathcal{H}_p$ for every $E \in \mathcal{H}_p$. 
Such a Dirichlet  series $D$ is called a multiplier of $\mathcal{H}_p$ and the mapping $M_D: \mathcal{H}_p \to \mathcal{H}_p$, given by $M_D(E)= D\cdot E$, is referred as its associated multiplication operator. 

In \cite{bayart2002hardy} (see also \cite{defant2018Dirichlet, hedenmalm1997hilbert,queffelec2013diophantine}) it is proved that the multipliers of $\mathcal{H}_p$ are precisely those Dirichlet series that belong to the Banach space $\mathcal{H}_\infty$. Moreover, for a multiplier $D$ we have the following equality:
\[
\Vert M_D \Vert_{\mathcal H_p \to \mathcal H_p} = \Vert D \Vert_{\mathcal H_{\infty}}.
\]

Given $1 \leq p, q \leq \infty$, we propose to study the multipliers of $\mathcal{H}_p$ to $\mathcal{H}_q$; that is, we want to understand those Dirichlet series $D$ which verify that $D\cdot E \in \mathcal{H}_q$ for every $E \in \mathcal{H}_p$. 

For this we use the relation that exists between the Hardy spaces of Dirichlet series and the Hardy spaces of functions. The mentioned connection is given by the so-called Bohr lift $\mathcal{L}$, which identifies each Dirichlet series with a function (both in the polytorus and in the polydisk; see below for more details).

This identification allows us to relate the multipliers in spaces of Dirichlet series with those of function spaces. 
As consequence of our results, we obtain a complete characterization of $\mathfrak{M}(p,q)$, the space of multipliers of $\mathcal{H}_p$ into $\mathcal{H}_q$.
 It turns out that this set coincides with the Hardy space $\mathcal{H}_{pq/(p-q)}$ when $1\leq q<p \leq \infty$ and with the null space if $1 \leq p<q \leq \infty$. 
Precisely, for a multiplier $D \in \mathfrak{M}(p,q)$ where $1\leq q<p \leq \infty$ we have the isometric correspondence
\[
\Vert M_D \Vert_{\mathcal H_p \to \mathcal H_q} = \Vert D \Vert_{\mathcal H_{pq/(p-q)}}.
\]

Moreover, for certain values of $p$ and $q$ we study some structural properties of these multiplication operators.
Inspired by  some of the results obtained by Vukoti\'c \cite{vukotic2003analytic} and Demazeux \cite{demazeux2011essential} for spaces of holomoprhic functions in one variable, we get the corresponding version in the Dirichlet space context. 
In particular, when considering endomorphisms (i.e., $p=q$), the essential norm and the operator norm of a given multiplication operator coincides if $p>1$. In the remaining cases, that is $p=q=1$ or $1\leq q < p \leq \infty$, we compare the essential norm with the norm of the multiplier in different Hardy spaces.

We continue by studying the structure of the spectrum of  the multiplication operators over $\mathcal{H}_p$. Specifically, we consider the continuum spectrum, the radial spectrum and the approximate spectrum.
For the latter,  we use some necessary and sufficient conditions regarding the associated Bohr lifted function $\mathcal{L}(D)$ (see definition below) for which the multiplication operator $M_D : \mathcal H_p \to \mathcal{H}_p$ has closed range. 
\section{Preliminaries on Hardy spaces}

\subsection{Of holomorphic functions}

We note by $\mathbb{D}^{N} = \mathbb{D} \times \mathbb{D} \times \cdots$ the cartesian product of $N$ copies of the open unit disk $\mathbb{D}$ with $N\in \mathbb{N}\cup \{\infty\}$ and  $\mathbb{D}^{\infty}_{2}$ the domain in $\ell_2$ defined as $\ell_2 \cap \mathbb{D}^{\infty}$  (for coherence in the notation we will sometimes write $\mathbb{D}^N_2$ for $\mathbb{D}^N$ also in the case $N\in \mathbb{N}$).
We define $\mathbb{N}_0^{(\mathbb{N})}$ as consisting of all sequences $\alpha = (\alpha_{n})_{n}$ with $\alpha_{n} \in \mathbb{N}_{0} = \mathbb{N} \cup \{0\}$ which are eventually null. In this case we denote $\alpha ! := \alpha_1! \cdots \alpha_M!$ whenever $\alpha = (\alpha_1, \cdots, \alpha_M, 0,0,0, \dots)$.

A function $f: \mathbb{D}^{\infty}_2 \to \mathbb{C}$ is holomorphic if it is Fr\'echet differentiable at every $z\in \mathbb{D}^{\infty}_2$, that is, if there exists a continuous linear functional $x^*$ on $\ell_2$ such that 
\[
\lim\limits_{h\to 0} \frac{f(z+h)-f(z)- x^*(h)}{\Vert h \Vert}=0.
\]
We denote by $H_{\infty} (\mathbb{D}^{\infty}_2)$ the space of all bounded holomorphic functions $f : \mathbb{D}^\infty_2 \to \mathbb{C}$. 
For $1\leq  p< \infty$ we consider the Hardy spaces of holomorphic functions on the domain $\mathbb{D}^{\infty}_2$  defined by 
\begin{multline*}
  H_p(\mathbb{D}^\infty_2) :=\{ f : \mathbb{D}^\infty_2 \to \mathbb{C} : \; f \; \text{is holomorphic and } \\ \Vert f \Vert_{H_p(\mathbb{D}_2^\infty)} := \sup\limits_{M\in \mathbb{N}} \sup\limits_{ 0<r<1} \left( \int\limits_{\mathbb{T}^M} \vert f(r\omega, 0) \vert^p \mathrm{d}\omega \right)^{1/p} <\infty \}.  
\end{multline*}

The definitions of $H_{\infty} (\mathbb{D}^{N})$ and  $H_p(\mathbb{D}^{N})$ for finite $N$ are analogous (see  \cite[Chapters~13 and~15]{defant2018Dirichlet}).\\

For $N \in \mathbb{N} \cup \{ \infty \}$, each function $f\in H_p(\mathbb{D}^N_2)$ defines a unique family of coefficients $c_{\alpha}(f)= \frac{(\partial^{\alpha} f)(0)}{\alpha !}$ (the Cauchy coefficients) with $\alpha \in \mathbb{N}_0^{N}$ having always only finitely many non-null coordinates.
For $z \in \mathbb{D}^N_2$ one has the following monomial expansion \cite[Theorem~13.2]{defant2018Dirichlet} 
\[
f(z)= \sum\limits_{\alpha \in \mathbb{N}_0^{(\mathbb{N})}} c_{\alpha}(f) \cdot z^\alpha,
\] 
with $z^{\alpha} = z_1^{\alpha_1} \cdots z_M^{\alpha_M}$ whenever $\alpha = (\alpha_1, \cdots, \alpha_M, 0,0,0, \dots)$.\\

Let us note that for each fixed $N \in \mathbb{N}$ and $1 \leq p \leq \infty$ we have $H_{p}(\mathbb{D}^{N})  \hookrightarrow H_{p}(\mathbb{D}_{2}^{\infty})$ by doing $f \rightsquigarrow [ z = (z_{n})_{n} \in \mathbb{D}_{2}^{\infty} \rightsquigarrow f(z_{1}, \ldots z_{N})  ]$. Conversely, given a function $f \in H_{p}(\mathbb{D}_{2}^{\infty})$, for each $N \in \mathbb{N}$ we define $f_{N} (z_{1}, \ldots , z_{N}) = f (z_{1}, \ldots , z_{N}, 0,0, \ldots)$ 
for $(z_{1}, \ldots , z_{N}) \in \mathbb{D}^{N}$. It is well known  that $f_N \in H_p(\mathbb{D}^N)$.

An important property for our purposes  is the so-called Cole-Gamelin inequality (see \cite[Remark~13.14 and Theorem~13.15]{defant2018Dirichlet}), which states that for every $f\in H_p(\mathbb{D}^{N}_2)$ and $z \in \mathbb{D}^{N}_2$ (for $N \in \mathbb{N} \cup \{\infty\}$) we have
\begin{equation}\label{eq: Cole-Gamelin}
\vert f(z) \vert \leq \left( \prod\limits_{j=1}^{N} \frac{1}{1-\vert z_j \vert^2} \right)^{1/p} \Vert f \Vert_{H_p(\mathbb{D}^N_2)}.
\end{equation}
For functions of finitely many variable this inequality is optimal in the sense that if $N\in \mathbb{N}$ and $z\in \mathbb{D}^N$, then there is a function $f_z \in H_p(\mathbb{D}^N_2)$ given by
\begin{equation} \label{optima}
f_z(u) = \left( \prod\limits_{j=1}^N \frac{1- \vert z_j\vert^2}{(1- \overline{z}_ju_j)^2}\right)^{1/p},
\end{equation}
such that $\Vert f_z \Vert_{H_p(\mathbb{D}^N_2)} = 1$ and $\vert f_z(z) \vert = \left( \prod\limits_{j=1}^N \frac{1}{1-\vert z_j \vert^2} \right)^{1/p}$.

\subsection{On the polytorus}

On $\mathbb{T}^\infty = \{ \omega = ( \omega_{n})_{n} \colon \vert \omega_{n} \vert =1, \text{ for every } n  \}$ consider the product of the normalized Lebesgue measure on $\mathbb{T}$ (note that this is the Haar measure). For each $F \in L_1(\mathbb{T}^\infty)$ and $\alpha \in \mathbb{Z}^{(\mathbb{N})}$, the $\alpha-$th Fourier coefficient of $F$ is defined as
\[
\hat{F}(\alpha) = \int\limits_{\mathbb{T}^N} f(\omega) \cdot \omega^{\alpha} \mathrm{d}\omega 
\]
where again $\omega^{\alpha} = \omega_1^{\alpha_1}\cdots \omega_M^{\alpha_M}$ if $\alpha = (\alpha_{1}, \ldots , \alpha_{M}, 0,0,0, \ldots)$. 
The Hardy space on the polytorus $H_p(\mathbb{T}^\infty)$ is the subspace of $L_p(\mathbb{T}^\infty)$ given by all the functions $F$ such that $\hat{F}(\alpha)=0$ for every $\alpha \in \mathbb{Z}^{(\mathbb{N})} - \mathbb{N}_0^{(\mathbb{N})}$. The definition of $H_{p} (\mathbb{T}^{N})$ for finite $N$ is analogous (note that these are the classical Hardy spaces, see \cite{rudin1962fourier}).
We have the canonical inclusion $H_{p}(\mathbb{T}^{N})  \hookrightarrow H_{p}(\mathbb{T}^{\infty})$ by doing $F \rightsquigarrow [ \omega = (\omega_{n})_{n} \in \mathbb{T}^{\infty} \rightsquigarrow F(\omega_{1}, \ldots \omega_{N})  ]$.\\

Given $N_1 < N_2 \leq \infty$ and $F\in H_p(\mathbb{T}^{N_2})$, then the function $F_{N_1}$, defined by $F_{N_1}(\omega)= \int\limits_{\mathbb{T}^{N_2-N_1}} F(\omega,u)\mathrm{d}u$ for every $\omega\in \mathbb{T}^{N_1}$, belongs to $H_{p}(\mathbb{T}^{N_1})$. In this case, the Fourier coefficients of both functions coincide: that is, given $\alpha \in \mathbb{N}_0^{N_1}$ then 
\[
\hat{F}_{N_1}(\alpha)= \hat{F}(\alpha_1, \alpha_2, \dots, \alpha_{N_1},0,0, \dots).
\]
Moreover, 
\begin{equation*}
\Vert F \Vert_{H_p(\mathbb{T}^{N_2})} \geq \Vert F_{N_1} \Vert_{H_p(\mathbb{T}^{N_1})}.
\end{equation*}
Let $N \in \mathbb{N} \cup \{\infty\}$, there is an isometric isomorphism between the spaces $H_{p}(\mathbb{D}^N_2)$ and $H_p(\mathbb{T}^N)$. More precisely, given a function $f\in H_p(\mathbb{D}^N_2)$  there is a unique function $F\in H_p(\mathbb{T}^N)$ such that $c_{\alpha}(f) = \hat{F}(\alpha)$ for every $\alpha$ in the corresponding indexing set and $\Vert f \Vert_{H_{p}(\mathbb{D}^N_2)} =\Vert F \Vert_{H_p(\mathbb{T}^N)}$. If this is the case, we say that the functions $f$ and $F$ are associated. In particular, by the uniqueness of the coefficients, $f_{M}$ and $F_{M}$ are associated to each other for every $1 \leq M \leq  N$.
Even more, if $N\in \mathbb{N}$, then 
\[
F(\omega) = \lim\limits_{r\to 1^-} f(r\omega),
\]
for almost all $\omega \in \mathbb{T}^N$.

\noindent We isolate the following important property which will be useful later. 

\begin{remark} \label{manon}
Let $F \in H_p(\mathbb{T}^\infty)$.  If $1 \leq p < \infty$, then $F_{N} \to F$ in $H_{p}(\mathbb{T}^{\infty})$ (see e.g \cite[Remark~5.8]{defant2018Dirichlet}). If $p=\infty$, the convergence is given in the $w(L_{\infty},L_1)$-topology. 
 In particular, for any $1 \leq p \leq \infty$, there is a subsequence so that $\lim_{k} F_{N_{k}} (\omega) = F(\omega)$ for almost $\omega \in \mathbb{T}^{\infty}$ (note that the case $p=\infty$ follows directly from the inclusion $H_{\infty}(\mathbb{T}^\infty) \subset H_2(\mathbb{T}^\infty)$).
\end{remark}

\subsection{Bohr transform}

We previously mentioned the Hardy spaces of functions both on the polytorus and on the polydisk and the relationship between them based on their coefficients. 
This relation also exists with the Hardy spaces of Dirichlet series and the isometric isomorphism that identifies them is the so-called Bohr transform. 
To define it, let us first consider $\mathfrak{p}= (\mathfrak{p}_1, \mathfrak{p}_2, \cdots)$ the sequence of prime numbers. Then, given a natural number $n$, by the prime number decomposition, there are unique  non-negative integer numbers $\alpha_1, \dots , \alpha_M$ such that $n= \mathfrak{p}_1^{\alpha_1}\cdots \mathfrak{p}_M^{\alpha_M}$. 
Therefore, with the notation that we already defined, we have that $n= \mathfrak{p}^{\alpha}$ with $\alpha = (\alpha_1, \cdots, \alpha_M, 0,0, \dots)$.  Then, given $1\leq p \leq \infty$, the Bohr transform $\mathcal{B}_{\mathbb{D}^\infty_2}$ on $H_p(\mathbb{D}^\infty_2)$ is defined as follows: 
\[
\mathcal{B}_{\mathbb{D}^\infty_2}(f) = \sum\limits_n a_n n^{-s},
\]
where $a_n= c_{\alpha}(f)$ if and only if $n= \mathfrak{p}^{\alpha}$. The Bohr transform is an isometric isomorphism between the spaces $H_p(\mathbb{D}^{\infty}_2)$ and $\mathcal{H}_p$ (see \cite[Theorem~13.2]{defant2018Dirichlet}). 

We denote by $\mathcal H^{(N)}$ the set of all Dirichlet series $\sum a_{n} n^{-s}$ that involve only the first $N$ prime numbers; that is $a_n=0$ if $\mathfrak{p}_i$ divides $n$ for some $i>N$.
We write $\mathcal{H}_p^{(N)}$ for the space $\mathcal H^{(N)} \cap \mathcal H_p$ (endowed with the norm in $\mathcal H_p$). Note that  the image of $H_{p} (\mathbb{D}^{N})$ (seen as a subspace of $H_p(\mathbb{D}^{\infty}_2)$ with the natural identification) through  $\mathcal{B}_{\mathbb{D}^\infty_2}$ is exactly $\mathcal{H}_p^{(N)}$.

The inverse of the Bohr transform, which sends the space $\mathcal{H}_p$ into the space $H_p(\mathbb{D}^{\infty}_2)$, is called the \textit{Bohr lift},  which  we denote by $\mathcal{L}_{\mathbb{D}^\infty_2}$.

With the same idea, the Bohr transform $\mathcal{B}_{\mathbb{T}^\infty}$ on the polytorus for $H_p(\mathbb{T}^\infty)$ is defined; that is, 
\[
\mathcal{B}_{\mathbb{T}^\infty}(F) = \sum\limits_n a_n n^{-s},
\]
where $a_n = \hat{F}(\alpha)$ if and only if $n = \mathfrak{p}^\alpha$. It is an isometric ismorphism between the spaces $H_p(\mathbb{T}^N)$ and $\mathcal{H}_p$. Its inverse is denoted by $\mathcal{L}_{\mathbb{T}^\infty}$.

In order to keep the notation as clear as possible we will carefully use the following convention: we will use capital letters (e.g., $F$, $G$, or $H$) to denote functions defined on the polytorus $\mathbb{T}^{\infty}$ and lowercase letters (e.g., $f$, $g$ or $h$) to represent functions defined on the polydisk $\mathbb{D}_2^\infty$.  If $f$ and $F$ are associated to each other (meaning that $c_{\alpha}(f)= \hat{F}(\alpha)$ for every $\alpha$), we will sometimes write $f \sim F$. With the same idea, if a function $f$ or $F$ is associated  through the Bohr transform to a Dirichlet series $D$, we will write $f \sim D$ or $F\sim D$.

\section{The space of multipliers}

As we mentioned above, our main interest is to describe the multipliers of the Hardy spaces of Dirichlet series. Let us recall again that a holomorphic function $\varphi$, defined on $\mathbb{C}_{1/2}$ is a $(p,q)$-multiplier of $\mathcal{H}_{p}$ if $\varphi \cdot D \in \mathcal{H}_{q}$ for every $D \in \mathcal{H}_{p}$. We denote the set of all such functions by $\mathfrak{M}(p,q)$. Since the constant  $\mathbf{1}$ function belongs to $\mathcal{H}_{p}$ we have
that, if $\varphi \in \mathfrak{M}(p,q)$, then necessarily $\varphi$ belongs to $\mathcal{H}_{q}$ and it can be represented by a Dirichlet series. So, we will use that the multipliers of $\mathcal{H}_{p}$ are 
precisely Dirichlet series. The set $\mathfrak{M}^{(N)}(p,q)$ is defined in the obvious way, replacing $\mathcal{H}_{p}$ and $\mathcal{H}_{q}$ by  $\mathcal{H}_{p}^{(N)}$ and $\mathcal{H}_{q}^{(N)}$. The same argument as above shows that $\mathfrak{M}^{(N)}(p,q) \subseteq \mathcal{H}_{q}^{(N)}$.\\

The set $\mathfrak{M}(p,q)$ is clearly a vector space. Each Dirichlet series $D \in \mathfrak{M}(p,q)$ induces a multiplication operator $M_D$ from $\mathcal{H}_p$ to $\mathcal{H}_q$, defined by $M_D(E)=D\cdot E$. By the continuity of the evaluation on each $s \in \mathbb{C}_{1/2}$ (see e.g. \cite[Corollary 13.3]{defant2018Dirichlet}), and the Closed Graph Theorem, $M_D$ is continuous. Then, the expression
\begin{equation} \label{normamult}
\Vert D \Vert_{\mathfrak{M}(p,q)} := \Vert M_{D}  \Vert_{\mathcal{H}_{p} \to \mathcal{H}_{q}},
\end{equation}
defines a norm on $\mathfrak{M}(p,q)$. Note that
\begin{equation} \label{aleluya}
\Vert D \Vert_{\mathcal{H}_{q}} = \Vert M_D(1) \Vert_{\mathcal{H}_{q}} \leq \Vert M_D \Vert_{\mathcal{H}_{p} \to \mathcal{H}_{q}} \cdot \Vert 1 \Vert_{\mathcal{H}_{q}} = \Vert D \Vert_{\mathfrak{M}(p,q)} \,,
\end{equation}
and the inclusions that we presented above are continuous. A norm on $\mathfrak{M}^{(N)}(p,q)$ is defined analogously. \\
Clearly, if $p_{1}< p_{2}$ or $q_{1} < q_{2}$, then
\begin{equation}\label{inclusiones}
\mathfrak{M}(p_{1}, q) \subseteq \mathfrak{M}(p_{2},q) \text{ and }
\mathfrak{M}(p, q_{2}) \subseteq \mathfrak{M}(p,q_{1}) \,,
\end{equation}
for fixed $p$ and $q$.

Given a Dirichlet series $D = \sum a_{n} n^{-s}$, we denote by $D_{N}$ the `restriction' to the first $N$ primes (i.e., we consider those $n$'s that involve, in its factorization, only the first $N$ primes). Let us be more precise. If  $n \in \mathbb{N}$, we write $\gpd (n)$ for the greatest prime divisor of $n$. That is, if $n = \mathfrak{p}_1^{\alpha_{1}} \cdots \mathfrak{p}_N^{\alpha_{N}}$ (with $\alpha_{N} \neq 0$) is the prime decomposition of $n$, then $\gpd(n) = \mathfrak{p}_{N}$. With this notation, $D_{N} := \sum_{\gpd(n) \leq \mathfrak{p}_N} a_{n} n^{-s}$. 

\begin{proposition} \label{hilbert}
Let $D = \sum a_{n} n^{-s}$ be a Dirichlet series and $1 \leq p,q \leq \infty$. Then $D \in \mathfrak{M}(p,q)$ if and only if $D_{N} \in \mathfrak{M}^{(N)}(p,q)$ for every $N \in \mathbb{N}$ and $\sup_{N} \Vert D_{N} \Vert_{\mathfrak{M}^{(N)}(p,q)} < \infty$.
\end{proposition}
\begin{proof}
Let us begin by noting that, if $n=jk$, then clearly $\gpd (n) \leq \mathfrak{p}_{N}$ if and only if $\gpd (j) \leq \mathfrak{p}_{N}$ and $\gpd (k) \leq \mathfrak{p}_{N}$. From this we deduce that, given any two Dirichlet series $D$ and $E$, we have $(DE)_{N}= D_{N} E_{N}$ 
for every $N \in \mathbb{N}$. \\
Take some Dirichlet series $D$ and suppose that $D \in \mathfrak{M}(p,q)$. Then, given $E \in \mathcal{H}_{p}^{(N)}$ we have $DE \in \mathcal{H}_{q}$, and $(DE)_{N} \in  \mathcal{H}_{q}^{(N)}$. But $(DE)_{N} = D_{N} E_{N} = D_{N} E$ 
and, since $E$ was arbitrary, $D_{N} \in \mathfrak{M}^{(N)}(p,q)$ for every $N$. On the other hand, if $E \in \mathcal{H}_{q}$, then $E_{N} \in \mathcal{H}_{q}^{(N)}$ and $\Vert E_{N} \Vert_{\mathcal{H}_q} \leq \Vert E \Vert_{\mathcal{H}_q}$ (see
\cite[Corollary~13.9]{defant2018Dirichlet}). This gives $\Vert D_{N} \Vert_{\mathfrak{M}^{(N)}(p,q)} \leq  \Vert D \Vert_{\mathfrak{M}(p,q)}$ for every $N$.\\Suppose now that $D$ is such that $D_{N} \in \mathfrak{M}^{(N)}(p,q)$ for every $N$ and $ \sup_{N} \Vert D_{N} \Vert_{\mathfrak{M}^{(N)}(p,q)} < \infty$ (let us call it $C$). Then, for each $E \in \mathcal{H}_{p}$ we have, by \cite[Corollary~13.9]{defant2018Dirichlet},
\[
\Vert (DE)_{N} \Vert_{\mathcal{H}_p} = \Vert D_{N} E_{N} \Vert_{\mathcal{H}_p} \leq  \Vert D_{N} \Vert_{\mathfrak{M}^{(N)}(p,q)} \Vert E_{N} \Vert_{\mathcal{H}_p} \leq C \Vert E \Vert_{\mathcal{H}_p} \,.
\]
Since this holds for every $N$, it shows (again by \cite[Corollary~13.9]{defant2018Dirichlet}) that $DE \in \mathcal{H}_{p}$ and completes the proof.
\end{proof}

We are going to exploit the connection between Dirichlet series and power series in infinitely many variables. This leads us to consider spaces of multipliers on Hardy spaces of functions. If $U$ is either $\mathbb{T}^{N}$ or $\mathbb{D}_{2}^{N}$ 
(with $N \in \mathbb{N} \cup \{\infty\}$) we consider the corresponding Hardy spaces $H_{p}(U)$ (for $1 \leq p \leq \infty$), and say that a function $f$ defined on $U$ is a $(p,q)$-multiplier of $H_{p}(U)$ if $ f \cdot g \in H_{q}(U)$ for every 
$f \in H_{p}(U)$. We denote the space of all such fuctions by $\mathcal{M}_{U}(p,q)$. The same argument as before with the constant $\mathbf{1}$ function shows that $\mathcal{M}_{U} (p,q) \subseteq H_{q}(U)$. 
Also, each multiplier defines a multiplication operator $M : H_{p}(U) \to H_{q}(U)$ which, by the Closed Graph Theorem, is continuous, and the norm of the operator defines a norm on the space of multipliers, as in \eqref{normamult}.\\

Our first step is to see that the identifications that we have just shown behave `well' with the multiplication, in the sense that whenever two pairs of functions are identified to each other, then so also are the products. Let us make a precise statement.

\begin{theorem} \label{jonas}
Let $D,E \in \mathcal{H}_{1}$, $f,g \in H_{1} (\mathbb{D}_{2}^{\infty})$ and $F,G  \in H_{1} (\mathbb{T}^{\infty})$  so that $f \sim F \sim D$ and $g \sim G \sim E$. Then, the following are equivalent
\begin{enumerate}
\item \label{jonas1} $DE \in \mathcal{H}_{1}$
\item \label{jonas2} $fg \in H_{1} (\mathbb{D}_{2}^{\infty})$
\item \label{jonas3} $FG  \in H_{1} (\mathbb{T}^{\infty})$
\end{enumerate}
and, in this case $DE \sim fg \sim FG$.
\end{theorem}

The equivalence between~\ref{jonas2} and~\ref{jonas3} is based in the case for finitely many variables.
\begin{proposition} \label{nana}
Fix $N \in \mathbb{N}$ and let $f,g \in H_{1} (\mathbb{D}^{N})$ and $F,G  \in H_{1} (\mathbb{T}^{N})$  so that $f \sim F$ and $g \sim G$. Then, the following are equivalent
\begin{enumerate}
\item\label{nana2}  $fg \in H_{1} (\mathbb{D}^{N})$
\item\label{nana3} $FG  \in H_{1} (\mathbb{T}^{N})$
\end{enumerate}
and, in this case, $fg \sim FG$.
\end{proposition}
\begin{proof}
Let us suppose first that $fg  \in H_{1} (\mathbb{D}^{N})$ and denote by $H \in H_{1} (\mathbb{T}^{N})$ the associated function. Then, since 
\[
F(\omega) = \lim_{r \to 1^{-}} f(r \omega) , \text{ and } G(\omega) = \lim_{r \to 1^{-}} g(r \omega) \,
\]
for almost all $\omega \in \mathbb{T}^{N}$, we have
\[
H (\omega) = \lim_{r \to 1^{-}} (fg)(r\omega) = F(\omega) G(\omega)
\]
for almost all $\omega \in \mathbb{T}^{N}$. Therefore $F G = H \in H_{1}(\mathbb{T}^{N})$, and this yields~\ref{nana3}. \\
Let us conversely assume that $FG \in H_{1}(\mathbb{T}^{N})$, and take the associated function $h \in H_{1} (\mathbb{D}^{N})$. The product $fg : \mathbb{D}^{N} \to \mathbb{C}$ is a holomorphic function and $fg -h$ belongs to the 
Nevanlinna class $\mathcal{N}(\mathbb{D}^{N})$, that is 
\[
\sup_{0<r<1} \int\limits_{\mathbb{T}^{N}} \log^{+} \vert f (r\omega) g(r\omega) - h(r\omega) \vert \mathrm{d} \omega < \infty \,
\]
where $\log^{+}(x):= \max \{0, \log x\}$ (see \cite[Section~3.3]{rudin1969function} for a complete account on this space). Consider $H(\omega)$ defined for almost all $\omega \in \mathbb{T}^{N}$ as the radial limit of $fg-h$. Then by \cite[Theorem 3.3.5]{rudin1969function} there are two possibilities: either $\log \vert H \vert \in L_{1}(\mathbb{T}^{N})$ or $fg-h =0$ on $\mathbb{D}^{N}$. But, just as before, we have
\[
\lim_{r \to 1^{-}} f(r\omega) g(r\omega) = F(\omega) G(\omega) = \lim_{r \to 1^{-}} h(r\omega)
\]
for almost all $\omega \in \mathbb{T}^{N}$, and then necessarily $H=0$. Thus $fg=h$ on $\mathbb{D}^{N}$, and $fg \in H_{1}(\mathbb{D}^{N})$. This shows that~\ref{nana3} implies~\ref{nana2} and completes the proof.
\end{proof}

For the general case we need the notion of the Nevanlinna class in the infinite dimensional framework. Given $\mathbb{D}_1^\infty := \ell_1 \cap \mathbb{D}^\infty$, a function $u: \mathbb{D}_1^\infty \to \mathbb{C}$ and $0< r < 1$, the mapping $u_{[r]}: \mathbb{T}^\infty \to \mathbb{C}$ is defined by 
\[
u_{[r]} (\omega) = (r\omega_1, r^2 \omega_2, r^3 \omega_3, \cdots).
\]
The Nevanlinna class on infinitely many variables,  introduced recently in \cite{guo2022dirichlet} and denoted by $\mathcal{N}(\mathbb{D}_1^\infty)$, consists on those holomorphic functions $u: \mathbb{D}_1^\infty \to \mathbb{C}$ such that 
\[
\sup\limits_{0<r<1} \int\limits_{\mathbb{T}^\infty} \log^+ \vert u_{[r]}(\omega)  \vert \mathrm{d} \omega < \infty.
\]
We can now prove the general case.

\begin{proof}[Proof of Theorem~\ref{jonas}]
Let us show first that~\ref{jonas1} implies~\ref{jonas2}. Suppose that $D=\sum a_{n} n^{-s}, E= \sum b_{n} n^{-s} \in \mathcal{H}_{1}$ are so that $\big(\sum a_{n} n^{-s} \big) \big( \sum b_{n} n^{-s}  \big) = 
\sum c_{n} n^{-s} \in \mathcal{H}_{1}$. Let $h \in H_{1}(\mathbb{D}_{2}^{\infty})$ be the holomorphic function associated to the product. Recall that, if $\alpha \in \mathbb{N}_{0}^{(\mathbb{N})}$ and $n = \mathfrak{p}^{\alpha} \in \mathbb{N}$, then
\begin{equation} \label{producto1}
c_{\alpha}(f) = a_{n} , \, c_{\alpha}(g) = b_{n} \text{ and } c_{\alpha} (h) = c_{n} = \sum_{jk=n} a_{j} b_{k} \,.
\end{equation}
On the other hand, the function $f \cdot g : \mathbb{D}_{2}^{\infty} \to \mathbb{C}$ is holomorphic and a straightforward computation shows that
\begin{equation} \label{producto2}
c_{\alpha} (fg) = \sum_{\beta + \gamma = \alpha} c_{\beta}(f) c_{\gamma}(g) \,.
\end{equation}
for every $\alpha$.
Now, if $jk=n = \mathfrak{p}^{\alpha}$ for some $\alpha  \in \mathbb{N}_{0}^{(\mathbb{N})}$, then there are $\beta, \gamma \in  \mathbb{N}_{0}^{(\mathbb{N})}$ so that $j = \mathfrak{p}^{\beta}$, $k = \mathfrak{p}^{\gamma}$ and $\beta + \gamma = \alpha$. 
This, together with \eqref{producto1} and \eqref{producto2} shows that $c_{\alpha}(h) = c_{\alpha} (fg)$ for every $\alpha$ and, therefore $fg=h \in H_{1} (\mathbb{D}_{2}^{\infty})$. This yields our claim.\\
Suppose now that  $fg  \in H_{1} (\mathbb{D}_{2}^{\infty})$ and take the corresponding Dirichlet series $\sum a_{n} n^{-s}$, $\sum b_{n} n^{-s}$, $\sum c_{n} n^{-s}  \in \mathcal{H}_{1}$ (associated to $f$, $g$  and $fg$ respectively). The same argument as above shows that
\[
c_{n} = c_{\alpha}(fg)= \sum_{\beta + \gamma = \alpha} c_{\beta}(f) c_{\gamma}(g) =  \sum_{jk=n} a_{j} b_{k} \, ,
\]
hence $\big(\sum a_{n} n^{-s} \big) \big( \sum b_{n} n^{-s}  \big) =  \sum c_{n} n^{-s} \in \mathcal{H}_{1}$, showing that~\ref{jonas2} implies~\ref{jonas1}.\\
Suppose now that $fg \in H_{1}(\mathbb{D}_{2}^{\infty})$ and let us see that~\ref{jonas3} holds. Let $H \in H_{1}(\mathbb{T}^{\infty})$ be the function associated to $fg$. Note first that $f_{N} \sim F_{N}$, $g_{N} \sim G_{N}$ and $(fg)_{N} \sim H_{N}$ for every $N$. A straightforward computation shows that $(fg)_{N} = f_{N} g_{N}$, and then this product is in $H_{1}(\mathbb{D}^{N})$. Then Proposition~\ref{nana} yields $f_{N} g_{N} \sim F_{N} G_{N}$, therefore
\[
\hat{H}_{N} (\alpha) = \widehat{(F_{N}G_{N})} (\alpha)
\]
for every $\alpha \in \mathbb{N}_{0}^{(\mathbb{N})}$ and, then, $H_{N} = F_{N}G_{N}$ for every $N \in \mathbb{N}$. We can find a subsequence in such a way that 
\[
\lim_{k} F_{N_{k}} (\omega) = F(\omega),  \, \lim_{k} G_{N_{k}} (\omega) = G(\omega),  \, \text{ and } \lim_{k} H_{N_{k}} (\omega) = H(\omega)
\]
for almost all $\omega \in \mathbb{T}^{\infty}$ (recall Remark~\ref{manon}). All this gives that $F(\omega)G(\omega) = H(\omega)$ for almost all $\omega \in \mathbb{T}^{\infty}$. Hence $FG = H \in H_{1} (\mathbb{T}^{\infty})$, and our claim is proved. \\
Finally, if $FG \in H_{1}(\mathbb{T}^{\infty})$, we denote by $h$ its associated function in $H_{1}(\mathbb{D}_{2}^{\infty})$. 
By \cite[Propostions~2.8 and 2.14]{guo2022dirichlet} we know that $H_1(\mathbb{D}_2^\infty)$ is contained in the Nevanlinna class $\mathcal{N}(\mathbb{D}_1^\infty)$,  therefore $f,g,h \in \mathcal{N}(\mathbb{D}_1^\infty)$ and hence, by definition, $f\cdot g - h \in \mathcal{N}(\mathbb{D}_1^\infty)$. On the other hand, \cite[Theorem~2.4 and Corollary~2.11]{guo2022dirichlet} tell us that, if $u \in \mathcal{N}(\mathbb{D}_1^\infty)$, then the radial limit $u^*(\omega) = \lim\limits_{r\to 1^-} u_{[r]} (\omega)$ exists for almost all $\omega\in \mathbb{T}^\infty$. Even more, $u=0$ if and only if $u^*$ vanishes on some subset of $\mathbb{T}^\infty$ with positive measure. The radial limit of $f,g$ and $h$ coincide a.e. with $F, G$ and $F\cdot G$ respectively (see \cite[Theorem~1]{aleman2019fatou}). Since 
\[
(f\cdot g - h)^* (\omega)= \lim\limits_{r\to 1^-} f_{[r]}(\omega) \cdot g_{[r]}(\omega) -h_{[r]}(\omega) = 0,
\]
for almost all $\omega\in \mathbb{T}^\infty$, then $f\cdot g =h$ on $\mathbb{D}_1^\infty$. Finally, since the set $\mathbb{D}_1^\infty$ is dense in $\mathbb{D}_2^\infty$, by the continuity of the functions we have that $f\cdot g \in H_1(\mathbb{D}_2^\infty).$
\end{proof}

As an immediate consequence of Theorem~\ref{jonas} we obtain the following.

\begin{proposition} \label{charite}
For every $1 \leq p, q \leq \infty$ we have
\[
\mathfrak{M}(p,q) = \mathcal{M}_{\mathbb{D}_{2}^{\infty}}(p,q) = \mathcal{M}_{\mathbb{T}^{\infty}}(p,q) \,,
\]
and
\[
\mathfrak{M}^{(N)}(p,q) = \mathcal{M}_{\mathbb{D}^{N}}(p,q) = \mathcal{M}_{\mathbb{T}^{N}}(p,q) \,,
\]
for every $N \in \mathbb{N}$, by means of the Bohr transform. 
\end{proposition}

Again (as in Proposition~\ref{hilbert}), being a multiplier can be characterized in terms of the restrictions (this follows immediately from Proposition~\ref{hilbert} and Proposition~\ref{charite}).

\begin{proposition}\label{remark multiplicadores}
\,
\begin{enumerate}

\item $f \in \mathcal{M}_{\mathbb{D}^{\infty}_2}(p,q)$ if and only if $f_N \in \mathcal{M}_{\mathbb{D}^N_2}(p,q)$ for every $N \in \mathbb{N}$ and $\sup_{N} \Vert M_{f_{N}} \Vert < \infty$.

\item $F \in \mathcal{M}_{\mathbb{T}^{\infty}}(p,q)$, then, $F_N \in \mathcal{M}_{\mathbb{T}^N}(p,q)$ for every $N \in \mathbb{N}$  and $\sup_{N} \Vert M_{F_{N}} \Vert < \infty$.
\end{enumerate}
\end{proposition}

The following statement describes the spaces of multipliers, viewing them as Hardy spaces of Dirichlet series.
A result of similar flavour for holomorphic functions in one variable appears in \cite{stessin2003generalized}. 

\begin{theorem}\label{descripcion}
The following assertions hold true
\begin{enumerate}
    \item \label{descr1} $\mathfrak{M}(\infty,q)= \mathcal{H}_q$ isometrically.
    \item \label{descr2}  If $1\leq q<p<\infty$ then
    $\mathfrak{M}(p,q) = \mathcal{H}_{pq/(p-q)} $ \; isometrically.
    \item \label{descr3} If $1 \leq p \leq \infty$ then $\mathfrak{M}(p,p)= \mathcal{H}_{\infty}$ isometrically.
    \item \label{descr4} If $1 \le p<q \leq \infty$ then $\mathfrak{M}(p,q)=\{0\}$.
\end{enumerate}
The same equalities hold if we replace in each case $\mathfrak{M}$ and $\mathcal{H}$ by $\mathfrak{M}^{(N)}$ and $\mathcal{H}^{(N)}$ (with $N \in \mathbb{N}$) respectively.
\end{theorem}

\begin{proof} 
To get the result we use again the isometric identifications between the Hardy spaces of Dirichlet series and both Hardy spaces of functions, and also between their multipliers given in Proposition~\ref{charite}.
Depending on each case we will use the most convenient identification, jumping from one to the other without further notification.

\ref{descr1} We already noted that  $\mathcal{M}_{\mathbb{T}^{N}}(\infty,q)\subset H_{q}(\mathbb{T}^N)$ with continuous inclusion (recall \eqref{aleluya}). On the other hand, if $D \in \mathcal{H}_{q}$ and $E \in \mathcal{H}_{\infty}$ then $D\cdot E$ a Dirichlet series in $\mathcal{H}_{q}$. Moreover,
\[
\Vert M_D(E) \Vert_{\mathcal{H}_{q}} 
\leq \Vert D \Vert_{\mathcal{H}_{q}} \Vert E \Vert_{\mathcal{H}_{\infty}}.
\]
This shows that $\Vert M_D \Vert_{\mathfrak{M}(\infty,q)} \leq \Vert D \Vert_{\mathcal{H}_{q}},$ providing the isometric identification.

\ref{descr2} Suppose $1 \leq q<p<\infty$ and take some $f \in H_{pq/(p-q)} (\mathbb{D}^\infty_2)$ and $g\in H_{p}(\mathbb{D}^\infty_2)$, then $f\cdot g$ is holomorphic on $\mathbb{D}^\infty_2$. Consider $t= \frac{p}{p-q}$ and note that $t$ is the conjugate exponent of $\frac{p}{q}$ in the sense that $\frac{q}{p} + \frac{1}{t} = 1$. Therefore given $M\in \mathbb{N}$ and $0< r <1$, by H\"older inequality
\begin{align*}
\left( \int\limits_{\mathbb{T}^M} \vert f\cdot g(r\omega,0) \vert^q \mathrm{d}\omega \right)^{1/q} & \leq \left( \int\limits_{\mathbb{T}^M} \vert f(r\omega, 0) \vert^{qt} \mathrm{d}\omega \right)^{1/qt}\left( \int\limits_{\mathbb{T}^M} \vert g(r\omega, 0) \vert^{qp/q} \mathrm{d}\omega \right)^{q/qp} \\
&= \left( \int\limits_{\mathbb{T}^M} \vert f(r\omega, 0) \vert^{qp/(p-q)} \mathrm{d}\omega \right)^{(p-q)/qp} \left( \int\limits_{\mathbb{T}^M} \vert g(r\omega, 0) \vert^p \mathrm{d}\omega \right)^{1/p} \\
&\leq \Vert f \Vert_{H_{pq/(p-q)}(\mathbb{D}^\infty_2)} \Vert g \Vert_{H_p(\mathbb{D}^\infty_2)}.
\end{align*}
Since this holds for every $M\in \mathbb{N}$ and $0<r<1$, then $f\in \mathcal{M}_{\mathbb{D}^\infty_2}(p,q)$ and furthermore $\Vert M_f \Vert_{\mathcal{M}_{\mathbb{D}^\infty_2}(p,q)} \leq \Vert f \Vert_{H_{pq/(p-q)}(\mathbb{D}^\infty_2)},$. Thus $H_{pq/(p-q)} (\mathbb{D}^\infty_2) \subseteq \mathcal{M}_{\mathbb{D}^\infty_2}(p,q)$. The case for $\mathbb{D}^{N}$ with $N\in\mathbb{N}$ follows with the same idea.\\

To check that the converse inclusion holds, take some $F \in \mathcal{M}_{\mathbb{T}^N}(p,q)$ (where $N \in \mathbb{N} \cup \{\infty\}$) and consider the associated multiplication operator $M_F : H_p(\mathbb{T}^N) \to H_{q}(\mathbb{T}^N)$ which, as we know, is continuous. Let us see that it can be extended to a continuous operator on $L_{q}(\mathbb{T}^{N})$. To see this, take a trigonometric polynomial $Q$, that is a finite sum of the form
\[
Q(z)=\sum\limits_{\vert \alpha_i\vert \leq k} a_{\alpha} z^{\alpha} \,,
\]
and note that 
\begin{equation} \label{desc polinomio}
Q= \left( \prod\limits_{j=1}^{M} z_{j}^{-k} \right) \cdot P,
\end{equation}
where $P$ is the polynomial defined as $P:= \sum\limits_{0\leq \beta_i \leq 2k} b_{\beta} z^{\beta}$ and $b_{\beta}= a_{\alpha}$ whenever $\beta = \alpha +(k,\cdots, k, 0)$. Then,
\begin{align*}
\left(\int\limits_{\mathbb{T}^N} \vert F\cdot Q(\omega)\vert^q \mathrm{d}\omega\right)^{1/q} &= \left(\int\limits_{\mathbb{T}^N} \vert F\cdot P(\omega)\vert^q  \prod\limits_{j=1}^{M} \vert \omega_{j}\vert^{-kq} \mathrm{d}\omega\right)^{1/q} = \left(\int\limits_{\mathbb{T}^N} \vert F\cdot P(\omega)\vert^q \mathrm{d}\omega\right)^{1/q} \\
&\leq C \Vert P \Vert_{H_p(\mathbb{T}^N)} = C \left(\int\limits_{\mathbb{T}^N} \vert  P(\omega)\vert^p  \prod\limits_{j=1}^{M} \vert \omega_{j}\vert^{-kp} \mathrm{d}\omega\right)^{1/p} \\
&= C \Vert Q \Vert_{H_p(\mathbb{T}^N)}.
\end{align*}
Consider now an arbitrary $H\in L_p(\mathbb{T}^N)$ and, using \cite[Theorem~5.17]{defant2018Dirichlet} find  a sequence of trigonometric polynomials  $(Q_n)_n$
such that $Q_n \to H$ in $L_p$ and also a.e. on $\mathbb{T}^N$ (taking a subsequence if necessary). We have 
\[
\Vert F\cdot Q_n - F \cdot Q_m \Vert_{H_q(\mathbb{T}^N)} =\Vert F\cdot (Q_n-Q_m) \Vert_{H_q(\mathbb{T}^N)} \leq C \Vert Q_n - Q_m \Vert_{H_p(\mathbb{T}^N)} \to 0
\]
which shows that $(F\cdot Q_n)_n$ is a Cauchy sequence in $L_q(\mathbb{T}^N)$. Since $F\cdot Q_n \to F\cdot H$ a.e. on $\mathbb{T}^N$, then this proves that $F\cdot H \in L_q (\mathbb{T}^N)$ and  $F\cdot Q_n \to F\cdot H$ in $L_q(\mathbb{T}^N)$. 
Moreover, 
\[
\Vert F\cdot H \Vert_{H_q(\mathbb{T}^N)} = \lim \Vert F\cdot Q_n \Vert_{H_q(\mathbb{T}^N)} \leq C \lim \Vert Q_n \Vert_{H_p(\mathbb{T}^N)} = C \Vert H \Vert_{H_p(\mathbb{T}^N)},
\]
and therefore the operator $M_F : L_p(\mathbb{T}^N) \to L_q (\mathbb{T}^N)$ is well defined and bounded. In particular, $\vert F \vert^q \cdot \vert H\vert^q \in L_1(\mathbb{T}^N)$ for every $H\in L_p(\mathbb{T}^N)$.

Now, consider $H\in L_{p/q}(\mathbb{T}^N)$ then $\vert H\vert^{1/q} \in L_{p} (\mathbb{T}^N)$ and $\vert F\vert^q \cdot \vert H\vert \in L_1(\mathbb{T}^N)$ or, equivalently,  $\vert F\vert^q \cdot H \in L_1(\mathbb{T}^N)$. Hence 
\[
\vert F \vert^q \in L_{p/q}(\mathbb{T}^N)^* = L_{p/(p-q)}(\mathbb{T}^N),
\] 
and therefore $F\in L_{pq/(p-q)}(\mathbb{T}^N)$.
To finish the argument, since $\hat{F}(\alpha)=0$ whenever $\alpha \in \mathbb{Z}^N \setminus \mathbb{N}_{0}^N$ then $F\in H_{pq/(p-q)}(\mathbb{T}^N)$. We then conclude that
\[
H_{pq/(p-q)}( \mathbb{T}^N) \subseteq \mathcal{M}_{\mathbb{T}^{N}}(p,q) \,.
\]
In order to see the isometry, given $F\in H_{pq/(p-q)}(\mathbb{T}^N)$ and let $G=\vert F \vert^r \in L_p(\mathbb{T}^N)$ with $r = q/(p-q)$ then $F\cdot G \in L_q(\mathbb{T}^N)$. Let $Q_n$ a sequence of trigonometric polynomials such that $Q_n \to G$ in $L_p(\mathbb{T}^N)$, since $M_F: L_p(\mathbb{T}^N) \to L_q(\mathbb{T}^N)$ is continuous then $F\cdot Q_n = M_F(Q_n) \to F\cdot G$. On the other hand, writing $Q_n$ as \eqref{desc polinomio} we have for each $n\in \mathbb{N}$ a polynomial $P_n$ such that $\Vert F\cdot Q_n \Vert_{L_q(\mathbb{T}^N)} = \Vert F \cdot P_n \Vert_{L_q(\mathbb{T}^N)}$ and $\Vert Q_n \Vert_{L_p(\mathbb{T}^N)} = \Vert P_n \Vert_{L_p(\mathbb{T}^N)}$. Then we have that

\begin{multline*}
\Vert F \cdot G \Vert_{L_q(\mathbb{T}^N)} = \lim\limits_n \Vert F \cdot Q_n \Vert_{L_q(\mathbb{T}^N)} = \lim\limits_n \Vert F \cdot P_n \Vert_{L_q(\mathbb{T}^N)} 
\leq \lim\limits_n \Vert M_F \Vert_{\mathcal{M}_{\mathbb{T}^{N}}(p,q)} \Vert P_n \Vert_{L_p(\mathbb{T}^N)}  \\= \lim\limits_n \Vert M_F \Vert_{\mathcal{M}_{\mathbb{T}^{N}}(p,q)} \Vert Q_n \Vert_{L_p(\mathbb{T}^N)} = \Vert M_F \Vert_{\mathcal{M}_{\mathbb{T}^{N}}(p,q)} \Vert G \Vert_{L_p(\mathbb{T}^N)}.
\end{multline*}
Now, since 
\[
\Vert F \Vert_{L_{pq/(p-q)}(\mathbb{T}^N)}^{p/(p-q)} = \Vert F^{r + 1} \Vert_{L_q(\mathbb{T}^N)} = \Vert F \cdot G \Vert_{L_q(\mathbb{T}^N)}
\]
and 
\[
\Vert F \Vert_{L_{pq/(p-q)}(\mathbb{T}^N)}^{q/(p-q)} = \Vert F^{r} \Vert_{L_p(\mathbb{T}^N)} = \Vert G \Vert_{L_p(\mathbb{T}^N)}
\]
then 
\[
\Vert M_F \Vert_{\mathcal{M}_{\mathbb{T}^{N}}(p,q)} \geq \Vert F \Vert_{L_{pq/(p-q)}}= \Vert F \Vert_{H_{pq/(p-q)}(\mathbb{T}^N)},
\]
as we wanted to show.
 
\ref{descr3} was proved in \cite[Theorem~7]{bayart2002hardy}.

We finish the proof by seeing that~\ref{descr4} holds. On one hand, the previous case and \eqref{inclusiones} immediately give 
the inclusion 
\[
\{0\} \subseteq \mathcal{M}_{\mathbb{T}^{N}}(p,q) \subseteq H_{\infty}(\mathbb{T}^N).
\]
We now show that $\mathcal{M}_{\mathbb{D}_{2}^{N}}(p,q)=\{0\}$ for any $N\in\mathbb{N} \cup \{\infty\}$. We consider in first place the case $N \in \mathbb{N}$. For $1 \leq p < q < \infty$, we fix $f \in \mathcal{M}_{\mathbb{D}^N_2}(p,q)$ and $M_{f}$ the associated multiplication operator from $H_p(\mathbb{D}^N)$ to $H_q(\mathbb{D}^N)$. Now, given $g\in H_{p}(\mathbb{D}^{N}_2)$,  by \eqref{eq: Cole-Gamelin} we have
\begin{equation}\label{ec. desigualdad del libro}
\vert f\cdot g(z) \vert \leq \left( \prod\limits_{j=1}^N \frac{1}{1-\vert z_j \vert^2} \right)^{1/q} \Vert f\cdot g\Vert_{H_q(\mathbb{D}^N_2)} \leq \left( \prod\limits_{j=1}^N \frac{1}{1-\vert z_j \vert^2} \right)^{1/q} C \Vert g \Vert_{H_p(\mathbb{D}^N_2)}.
\end{equation}
Now since $f\in H_{\infty}(\mathbb{D}^N_2)$ and 
\[
\Vert f \Vert_{H_\infty(\mathbb{D}^N)} = \lim\limits_{r\to 1} \sup\limits_{z\in r\mathbb{D}^N_2} \vert f(z) \vert = \lim\limits_{r\to 1} \sup\limits_{z\in r\mathbb{T}^N} \vert f(z) \vert,
\]
then there is a sequence $(u_n)_n\subseteq \mathbb{D}^N$ such that $\Vert u_n \Vert_{\infty} \to 1$ and 
\begin{equation}\label{limite sucesion}
\vert f(u_n) \vert \to \Vert f \Vert_{H_\infty(\mathbb{D}^N_2)}.
\end{equation}
For each $u_n$ there is a non-zero function $g_n\in H_{p}(\mathbb{D}^N)$ (recall \eqref{optima}) such that
\[
\vert g_n(u_n) \vert = \left( \prod\limits_{j=1}^N \frac{1}{1-\vert u_n^j \vert^2} \right)^{1/p} \Vert g_n \Vert_{H_p(\mathbb{D}^N)}.
\]
From this and \eqref{ec. desigualdad del libro} we get
\[
\vert f(u_n) \vert \left( \prod\limits_{j=1}^N \frac{1}{1-\vert u_n^j \vert^2} \right)^{1/p} \Vert g_n \Vert_{H_p(\mathbb{D}^N)} \leq \left( \prod\limits_{j=1}^N \frac{1}{1-\vert u_n^j \vert^2} \right)^{1/q} C \Vert g_n \Vert_{H_p(\mathbb{D}^N)}.
\]
Then,
\[
\vert f(u_n) \vert \left( \prod\limits_{j=1}^N \frac{1}{1-\vert u_n^j \vert^2} \right)^{1/p-1/q} \leq C.
\]
Since $1/p-1/q>0$ we have that $\left( \prod\limits_{j=1}^N \frac{1}{1-\vert u_n^j \vert^2} \right)^{1/p-1/q} \to \infty,$ and then, by the previous inequality, $\vert f(u_n) \vert \to 0$. By \eqref{limite sucesion} this shows that $\Vert f \Vert_{H_\infty(\mathbb{D}^N)}=0$ and this gives the claim for $q<\infty$. Now if $q=\infty$, by noticing that $H_{\infty}(\mathbb{D}^N)$ is contained in $H_{t}(\mathbb{D}^N)$ for every $1 \leq p < t < \infty$ the result follows from the previous case. This concludes the proof for  $N \in \mathbb{N}$.\\
To prove that $\mathcal{M}_{\mathbb{D}^{\infty}_2}(p,q)=\{0\}$, fix again $f \in \mathcal{M}_{\mathbb{D}^{\infty}_2}(p,q).$
By Proposition~\ref{remark multiplicadores}, for every $N \in \mathbb{N}$ the truncated function $f_N \in \mathcal{M}_{\mathbb{D}^N_2}(p,q)$ and therefore, by what we have shown before, is the zero function. Now the proof follows using that $(f_{N})_{N}$ converges pointwise to $f$. \end{proof}

\section{Multiplication operator}

Given a multiplier $D \in \mathfrak{M}(p,q)$, we study in this section several properties of its associated multiplication operator $M_D : \mathcal{H}_p \to \mathcal{H}_q$. In \cite{vukotic2003analytic} Vukoti\'c provides a very complete description of certain  Toeplitz operators for Hardy spaces of holomorphic functions of one variable.  In particular he studies the spectrum, the range and the essential norm of these operators. Bearing in mind the relation between the sets of multipliers that we proved above (Proposition~\ref{charite}), it is natural to ask whether similar properties hold when we look at the multiplication operators on the Hardy spaces of Dirichlet series. 
%

In our first result we characterize which operators are indeed multiplication operators. These happen to be exactly those that commute with the monomials given by the prime numbers.

\begin{theorem}
Let $1\leq p,q \leq \infty$. A bounded operator $T: \mathcal{H}_p \to \mathcal{H}_q$ is a multiplication operator if and only if $T$ commutes with the multiplication operators $M_{\mathfrak{p}_i^{-s}}$ for every $i \in \mathbb{N}$.

The same holds if we replace in each case $\mathcal{H}$ by $\mathcal{H}^{(N)}$ (with $N \in \mathbb{N}$), and considering $M_{\mathfrak{p}_i^{-s}}$ with $1 \leq i \leq N$.
\end{theorem}

\begin{proof}
Suppose first that  $T: \mathcal{H}_p \to \mathcal{H}_q$ is a multiplication operator (that is, $T=M_D$ for some Dirichlet series $D$) and for $i \in \mathbb{N}$, let $\mathfrak{p}_i^{-s}$ be a monomial, then 
\[
T \circ M_{\mathfrak{p}_i^{-s}} (E)= D \cdot \mathfrak{p}_i^{-s} \cdot E= \mathfrak{p}_i^{-s} \cdot D \cdot E = M_{\mathfrak{p}_i^{-s}} \circ T (E).
\]
That is, $T$ commutes with $M_{\mathfrak{p}_i^{-s}}$.

For the converse, suppose now that $T: \mathcal{H}_p \to \mathcal{H}_q$ is a bounded operator that commutes with the multiplication operators $M_{\mathfrak{p}_i^{-s}}$ for every $i \in \mathbb{N}$. Let us see that $T = M_D$ with $D = T(1)$. Indeed, for each $\mathfrak{p}_i^{-s}$ and $k\in \mathbb{N}$ we have that 
\[
T((\mathfrak{p}_i^{k})^{-s})=T((\mathfrak{p}_i^{-s})^{k}) = T(M_{\mathfrak{p}_i^{-s}}^{k}(1)) = M_{\mathfrak{p}_i^{-s}}^{k}( T(1)) = (\mathfrak{p}_i^{-s})^{k} \cdot D = (\mathfrak{p}_i^{k})^{-s} \cdot D,
\]
and then given $n\in \mathbb{N}$ and $\alpha \in \mathbb{N}_0^{(\mathbb{N})}$ such that $n = \mathfrak{p}_1^{\alpha_1} \cdots \mathfrak{p}_k^{\alpha_k}$
\[
T(n^{-s})= T( \prod\limits_{j=1}^k (\mathfrak{p}_i^{\alpha_i})^{-s} ) = T ( M_{\mathfrak{p}_1^{-s}}^{\alpha_1} \circ \cdots \circ M_{\mathfrak{p}_k^{-s}}^{\alpha_k} (1) ) = M_{\mathfrak{p}_1^{-s}}^{\alpha_1} \circ \cdots \circ M_{\mathfrak{p}_k^{-s}}^{\alpha_k} ( T(1) ) = (n^{-s}) \cdot D.
\]
This implies that $T(P)= P \cdot D$ for every Dirichlet polynomial $P$. Take now some $E\in \mathcal{H}_p$ and choose a sequence of polynomials  $P_n$ that converges in norm to $E$ if $1 \leq p < \infty$ or weakly if $p= \infty$ (see \cite[Theorems~5.18 and~11.10]{defant2018Dirichlet}).
In any case, if $s \in \mathbb{C}_{1/2}$, the continuity of the evaluation at $s$  (see again \cite[Corollary~13.3]{defant2018Dirichlet}) yields  $P_n(s) \to E(s)$.
Since $T$ is continuous, we have that 
\[
T(E) = \lim\limits_n T(P_n)= \lim\limits_n P_n\cdot D
\]
(where the limit is in the weak topology if  $p=\infty$).
Then for each $s\in \mathbb{C}$ such that $\re s > 1/2$, we have 
\[
T(E)(s) = \lim\limits_n P_n\cdot D(s) = E(s) D(s).
\]
Therefore, $T(E) = D \cdot E$ for every Dirichlet series $E$. In other words, $T$ is equal to $M_D$, which concludes the proof. 
\end{proof}

Given a bounded operator $T: E \to F$ the essential norm is defined as 
\[
\Vert T \Vert_{\ess} = \inf \{ \Vert T - K \Vert : \; K : E \to F \; \text{ compact} \}.
\]
This norm tells us how far from being compact $T$ is.


The following result shows a series of comparisons between essential norm of $M_D : \mathcal{H}_p \to \mathcal{H}_q$ and the norm of $D$, depending on $p$ and $q$. In all cases, as a consequence, the operator is compact if and only if $D=0$.

\begin{theorem} \label{chatruc}
\;
\begin{enumerate}
\item\label{chatruc1} Let $1\leq q < p < \infty$, $D\in \mathcal{H}_{pq/(p-q)}$ and $M_D$ its associated multiplication operator from $\mathcal{H}_p$ to $\mathcal{H}_q$. Then 
\[
\Vert D \Vert_{\mathcal{H}_q} \leq \Vert M_D \Vert_{\ess} \leq \Vert M_D \Vert = \Vert D \Vert_{\mathcal{H}_{pq/(p-q)}}.
\]
\item \label{chatruc2}  Let $1\leq q < \infty$, $D\in \mathcal{H}_q$ and $M_D : \mathcal{H}_\infty \to \mathcal{H}_q$ the multiplication operator. Then 
\[
\frac{1}{2}\Vert D \Vert_{\mathcal{H}_q} \leq \Vert M_D \Vert_{\ess} \leq \Vert M_D \Vert = \Vert D \Vert_{\mathcal{H}_q}.
\]
\end{enumerate}
In particular, $M_D$ is compact if and only if $D=0$. The same equalities hold if we replace $\mathcal{H}$ by $\mathcal{H}^{(N)}$ (with $N \in \mathbb{N}$).
\end{theorem}

We start with a lemma based on \cite[Proposition~2]{brown1984cyclic} for Hardy spaces of holomorphic functions. We prove that weak-star convergence and uniformly convergence on half-planes are equivalent on Hardy spaces of Dirichlet series. We are going to use that $\mathcal{H}_{p}$ is a dual space for every $1 \leq p < \infty$. For $1<p<\infty$ this is obvious because the space is reflexive. For $p=1$ in \cite[Theorem~7.3]{defantperez_2018} it is shown, for Hardy spaces of vector valued Dirichlet series, that $\mathcal{H}_{1}(X)$ is a dual space if and only if $X$ has the Analytic Radon-Nikodym property. Since $\mathbb{C}$ has the ARNP, this gives what we need. We include here an alternative proof in more elementary terms.

\begin{proposition} \label{basile}
The space $\mathcal{H}_1$ is a dual space. 
\end{proposition}
\begin{proof}
Denote by $(B_{H_1}, \tau_0)$ the closed unit ball of $H_1(\mathbb{D}_2^\infty)$, endowed with the topology $\tau_0$ given by the uniform convergence on compact sets. Let us show that $(B_{H_1}, \tau_0)$ is a compact set. Note first that, given a compact $K\subseteq \ell_2$ and $\varepsilon >0$, there exists $j_0 \in \mathbb{N}$ such that $\sum\limits_{j\geq j_0}^\infty \vert z_j \vert^2 < \varepsilon$ for all $z\in K$ \cite[Page 6]{diestel2012sequences}.  Then, from Cole-Gamelin inequality~\eqref{eq: Cole-Gamelin},  the set
\[
\{f(K) : f \in B_{H_1} \} \subset \mathbb{C}
\]
is bounded for each compact set $K$. By Montel's theorem (see e.g. \cite[Theorem~15.50]{defant2018Dirichlet}), $(B_{H_1},\tau_0)$ is relatively compact.
We now show that $(B_{H_1}, \tau_0)$ is closed. Indeed, suppose now that $(f_\alpha) \subset B_{H_1}$ is a net that converges to $B_{H_1}$ uniformly on compact sets, then we obviously have 
\[
\int\limits_{\mathbb{T}^N} \vert f(r\omega,0,0, \cdots) \vert \mathrm{d} \omega 
\leq \int\limits_{\mathbb{T}^N} \vert f(r\omega,0,0, \cdots) -f_\alpha(r\omega,0,0, \cdots) \vert \mathrm{d} \omega + \int\limits_{\mathbb{T}^N} \vert f_\alpha(r\omega,0,0, \cdots) \vert \mathrm{d} \omega.
\]
Since the first term tends to $0$ and the second term is less than or equal to $1$ for every $N \in \mathbb{N}$ and every $0 < r <1$, then the limit function $f$ belongs to $B_{H_1}$. 
Thus, $(B_{H_1}, \tau_0)$ is compact. \\
We consider now the set of functionals
\[
\{ev_z: H_1(\mathbb{D}_2^\infty) \to \mathbb C : z \in \mathbb{D}_2^\infty\}.
\]
Note that the weak topology $w(H_1,E)$ is exactly the topology given by the pointwise convergence. Thus, since a priori $\tau_0$ is clearly a stronger topology than $w(H_1,E)$ we have that $(B_{H_1},w(H_1,E))$ is also compact.  
Since  $E$ separates points, by \cite[Theorem~1]{kaijser1977note}, $H_1(\mathbb{D}_2^\infty)$ is a dual space and hence, using the Bohr transform, $\mathcal{H}_1$ also is a dual space.
\end{proof}

\begin{lemma}\label{bastia}
Let $1\leq p <\infty$ and $(D_n) \subseteq \mathcal{H}_p$ then the following statements are equivalent
\begin{enumerate}
\item \label{bastia1} $D_n \to 0$ in the weak-star topology.
\item \label{bastia2} $D_n(s) \to 0$ for each $s\in \mathbb{C}_{1/2}$ and $\Vert D_n \Vert_{\mathcal{H}_p} \leq C$ for some $C<0$.
\item \label{bastia3} $D_n \to 0$ uniformly on each half-plane $\mathbb{C}_{\sigma}$ with $\sigma > 1/2$ and $\Vert D_n \Vert_{\mathcal{H}_p} \leq C$  for some $C<0$.
\end{enumerate}
\end{lemma}
\begin{proof}
The implication~\ref{bastia1} then~\ref{bastia2} is verified by the continuity of the evaluations in the weak-star topology, and because the convergence in this topology implies that the sequence is bounded.

Let us see that~\ref{bastia2} implies~\ref{bastia3}. Suppose not, then there exists $\varepsilon>0$, a subsequence $(D_{n_j})_j$ and a half-plane $\mathbb{C}_\sigma$ with $\sigma > 1/2$ such that $\sup\limits_{s \in \mathbb{C}_\sigma} \vert D_{n_j}(s) \vert \geq \varepsilon$. Since $D_{n_j} = \sum\limits_{m} a_m^{n_j} m^{-s}$ is uniformly bounded, by Montel's theorem for $\mathcal{H}_p$ (see \cite[Theorem~3.2]{defant2021frechet}), there exists $D = \sum\limits_{m} a_m m^{-s} \in \mathcal{H}_p$ such that
\[
\sum\limits_{m} \frac{a_m^{n_j}}{m^{\delta}} m^{-s} \to \sum\limits_{m} \frac{a_m}{m^{\delta}} m^{-s} \; \text{in} \; \mathcal{H}_p
\]
for every $\delta >0$. Given $s \in \mathbb{C}_{1/2}$, we write $s= s_0 + \delta$ with $\delta >0$ and $s_0 \in \mathbb{C}_{1/2}$, to have
\[
D_{n_j}(s) = \sum\limits_{m} a_m^{n_j} m^{-(s_0 + \delta)} = \sum\limits_{m} \frac{a_m^{n_j}}{m^{\delta}} m^{-s_0} \to \sum\limits_{m} \frac{a_m}{m^{\delta}} m^{-s_0} = D(s_0+\delta) = D(s).
\]
We conclude that $D=0$ and by Cole-Gamelin inequality for Dirichlet series (see \cite[Corollary~13.3]{defant2018Dirichlet}) we have
\begin{align*}
\varepsilon &\leq \sup\limits_{\re s > 1/2 + \sigma} \vert D_{n_j} (s) \vert  = \sup\limits_{\re s > 1/2 + \sigma/2} \vert D_{n_j} (s + \sigma/2) \vert \\
&=  \sup\limits_{\re s > 1/2 + \sigma/2} \vert \sum\limits_{m} \frac{a_m^{n_j}}{m^{\sigma/2}} m^{-s} \vert \leq \zeta( 2 \re s)^{1/p} \Bigg\Vert \sum\limits_{m} \frac{a_m^{n_j}}{m^{\sigma/2}} m^{-s} \Bigg\Vert_{\mathcal{H}_p}\\
&\leq  \zeta(1+ \sigma)^{1/p} \Bigg\Vert \sum\limits_{m} \frac{a_m^{n_j}}{m^{\sigma/2}} m^{-s} \Bigg\Vert_{\mathcal{H}_p} \to 0,
\end{align*}
 for every $\sigma >0$, which is a contradiction.

To see that~\ref{bastia3} implies~\ref{bastia1}, let $B_{\mathcal{H}_p}$ denote the closed unit ball of $\mathcal{H}_{1}$. Since for each $1 \leq p <\infty$ the space $\mathcal{H}_{p}$ is a dual space, by Alaouglu's theorem, $(B_{\mathcal{H}_p}, w^*)$ (i.e. endowed with the weak-star topology) is compact. On the other hand $(B_{\mathcal{H}_p}, \tau_{0})$ (that is, endowed with the topology of uniform convergence on compact sets) is a Hausdorff topological space. If we show that the identity $Id : (B_{\mathcal{H}_p}, w^*) \to (B_{\mathcal{H}_p}, \tau_{0})$ is continuous, then it is a homeomorphism and the proof is completed. To see this let us note first that $\mathcal{H}_p$ is separable (note that the set of Dirichlet polynomials with rational coefficients is dense in $\mathcal{H}_p$) and then  $(B_{\mathcal{H}_p}, w^*)$ is  metrizable (see \cite[Theorem~5.1]{conway1990course}). Hence it suffices to work with sequences. If a sequence $(D_{n})_{n}$ converges in $w^{*}$ to some $D$, then in particular $(D_{n}-D)_{n}$ $w^{*}$-converges to $0$ and, by what we just have seen, it converges uniformly on compact sets. This shows that $Id$ is continuous, as we wanted.
%
%
%
\end{proof}

Now we prove Theorem~\ref{chatruc}. The arguments should be compared with \cite[Propositions~4.3 and~5.5]{demazeux2011essential} where similar statements have been obtained for weighted composition operators for holomorphic functions of one complex variable.

\begin{proof}[Proof of Theorem~\ref{chatruc}]
\ref{chatruc1} By definition $\Vert M_D \Vert_{\ess} \leq \Vert M_D \Vert = \Vert D \Vert_{\mathcal{H}_{pq/(p-q)}}$. In order to see the lower bound,
for each $n \in \mathbb{N}$ consider the monomial $E_n= (2^n)^{-s} \in \mathcal{H}_p$. Clearly $\Vert E_n \Vert_{\mathcal{H}_p} =1$ for every $n$, and $E_n(s) \to 0$ for each $s\in \mathbb{C}_{1/2}$. Then, by Lemma~\ref{bastia}, $E_n\to 0$ in the weak-star topology.

Take now some compact operator $K: \mathcal{H}_p \to \mathcal{H}_q$ and note that, since $\mathcal{H}_p$ is reflexive, we have $K(E_n) \to 0$, and hence
\begin{align*}
\Vert M_D -K \Vert \geq \limsup\limits_{n\to \infty} \Vert M_D(E_n) & - K(E_n) \Vert_{\mathcal{H}_q} \\
& \geq \limsup\limits_{n\to \infty} \Vert D\cdot E_n \Vert_{\mathcal{H}_q} -\Vert K(E_n) \Vert_{\mathcal{H}_q} = \Vert D \Vert_{\mathcal{H}_q}.
\end{align*}

\ref{chatruc2} Let $K: \mathcal{H}_\infty \to \mathcal{H}_q$ be a compact operator, and take again $E_n= (2^n)^{-s} \in \mathcal{H}_\infty$ for each $n\in \mathbb{N}$. Since $\Vert E_n \Vert_{\mathcal{H}_\infty} =1$ then there exists a subsequence $(E_{n_j})_j$ such that $(K(E_{n_j}))_j$ converges in $\mathcal{H}_q$. Given $\varepsilon > 0$ there exists $m\in \mathbb{N}$ such that if $j,l \geq m$ then 
\[
\Vert K(E_{n_j})-K(E_{n_l}) \Vert_{\mathcal{H}_q} < \varepsilon.
\]
On the other hand, if $D=\sum a_k k^{-s}$ then $D\cdot E_{n_l}= \sum a_k (k\cdot 2^{n_l})^{-s}$ and by \cite[Proposition~11.20]{defant2018Dirichlet} the norm in $\mathcal{H}_q$ of 
\[
(D\cdot E_{n_l})_\delta = \sum \frac{a_k}{(k\cdot 2^{n_l})^{\delta}} (k\cdot 2^{n_l})^{-s}
\]
tends increasingly to $\Vert D \cdot E_{n_l}\Vert_{\mathcal{H}_q} = \Vert D \Vert_{\mathcal{H}_q}$ when $\delta \to 0$. Fixed $j\geq m$, there exists $\delta >0$ such that 
\[
\Vert (D\cdot E_{n_j})_\delta \Vert_{\mathcal{H}_q} \geq \Vert D \Vert_{\mathcal{H}_q} - \varepsilon.
\]
Given that $\Vert \frac{E_{n_j} - E_{n_l}}{2} \Vert_{\mathcal{H}_\infty} = 1$ for every $j \not= l$ then
\begin{align*}
\Vert M_D - K \Vert & \geq \Bigg\Vert (M_D -K)  \frac{E_{n_j} - E_{n_l}}{2} \Bigg\Vert_{\mathcal{H}_q} \\
&\geq \frac{1}{2} \Vert (D \cdot E_{n_j} - D \cdot E_{n_l})_{\delta} \Vert_{\mathcal{H}_q} - \frac{1}{2} \Vert K(E_{n_j})-K(E_{n_l}) \Vert_{\mathcal{H}_q} \\
& >\frac{1}{2} (\Vert (D \cdot E_{n_j})_{\delta} \Vert_{\mathcal{H}_q} - \Vert (D \cdot E_{n_l})_{\delta} \Vert_{\mathcal{H}_q}) - \varepsilon/2 \\
& \geq \frac{1}{2} \Vert D \Vert_{\mathcal{H}_q} - \frac{1}{2} \Vert (D \cdot E_{n_l})_{\delta} \Vert_{\mathcal{H}_q} - \varepsilon.
\end{align*}
Finally, since 
\[
\Vert (D \cdot E_{n_l})_{\delta} \Vert_{\mathcal{H}_q} \leq \Vert D_\delta \Vert_{\mathcal{H}_q} \Vert (E_{n_l})_{\delta} \Vert_{\mathcal{H}_\infty} \leq \Vert D_\delta \Vert_{\mathcal{H}_q} \Vert \frac{(2^{n_l})^{-s}}{2^{n_l \delta}} \Vert_{\mathcal{H}_\infty} 
= \Vert D_\delta \Vert_{\mathcal{H}_q} \cdot \frac{1}{2^{n_l \delta}}, 
\]
and the latter tends to $0$ as $l \to \infty$, we finally have $\Vert M_D -K \Vert \geq \frac{1}{2} \Vert D \Vert_{\mathcal{H}_q}$.
\end{proof}

In the case of endomorphism, that is $p=q$, we give the following bounds for the essential norms.

\begin{theorem}\label{saja}
Let  $D\in \mathcal{H}_\infty$ and $M_D : \mathcal{H}_p \to \mathcal{H}_p$ the associated multiplication operator.
\begin{enumerate}
\item\label{saja1} If $1 < p \leq \infty$, then 
\[
\Vert M_D \Vert_{\ess} = \Vert M_D \Vert = \Vert D \Vert_{\mathcal{H}_\infty}.
\]

\item\label{saja2} If $p=1$, then
\[
\max\{\frac{1}{2}\Vert D \Vert_{\mathcal{H}_\infty} \; , \; \Vert D \Vert_{\mathcal{H}_1} \} \leq \Vert M_D \Vert_{\ess} \leq \Vert M_D \Vert = \Vert D \Vert_{\mathcal{H}_\infty}.
\]
\end{enumerate}
In particular, $M_D$ is compact if and only if $D=0$. The same equalities hold if we replace $\mathcal{H}$ by $\mathcal{H}^{(N)}$, with $N \in \mathbb{N}$.
\end{theorem}

The previous theorem will be a consequence of the Proposition~\ref{ubeda} which we feel is independently interesting. For the proof we need the following technical lemma in the spirit of \cite[Proposition~2]{brown1984cyclic}. Is relates weak-star convergence and uniform convergence on compact sets for Hardy spaces of holomorphic functions. It is a sort of `holomorphic version´ of Lemma~\ref{bastia}.

\begin{lemma}\label{maciel}
Let $1\leq p <\infty$, $N\in \mathbb{N}\cup \{\infty\}$ and $(f_n) \subseteq H_p(\mathbb{D}^N_2)$ then the following statements are equivalent
\begin{enumerate}
\item\label{maciel1} $f_n \to 0$ in the weak-star topology,
\item\label{maciel2} $f_n(z) \to 0$ for each $z\in \mathbb{D}^N_2$ and $\Vert f_n \Vert_{H_p(\mathbb{D}^N_2)} \leq C$
\item\label{maciel3} $f_n \to 0$ uniformly on compact sets of $\mathbb{D}^N_2$ and $\Vert f_n \Vert_{H_p(\mathbb{D}^N_2)} \leq C$,
\end{enumerate}
\end{lemma}
\begin{proof}
\ref{maciel1} $\Rightarrow$~\ref{maciel2} and~\ref{maciel3} $\Rightarrow$~\ref{maciel1} are proved with the same arguments used in Lemma~\ref{bastia}. 
Let us see~\ref{maciel2} $\Rightarrow$~\ref{maciel3}. Suppose not, then there exists $\varepsilon>0$, a subsequence $f_{n_j}$ and a compact set  $K \subseteq \mathbb{D}_{2}^{\infty}$ such that $\Vert f_{n_j}\Vert_{H_{\infty}(K)} \geq \varepsilon$. Since $f_{n_j}$ is bounded, by Montel's theorem for $H_p(\mathbb{D}^N_2)$ (see \cite[Theorem~2]{vidal2020montel}), we can take a subsequence $f_{n_{j_l}}$ and $f\in H_p(\mathbb{D}^N_2)$ such that $f_{n_{j_l}} \to f$ uniformly on compact sets. But since it tends pointwise to zero, then $f=0$ which is a contradiction.
\end{proof}

\begin{proposition}\label{ubeda}
\;
Let $1\leq p < \infty$, $f\in H_{\infty}(\mathbb{D}^\infty_2)$ and $M_f : H_p(\mathbb{D}^\infty_2) \to H_p(\mathbb{D}^\infty_2)$ the multiplication operator. If $p>1$ then
\[
\Vert M_f \Vert_{\ess} = \Vert M_f \Vert = \Vert f \Vert_{H_{\infty}(\mathbb{D}^\infty_2)}.
\]
If $p=1$ then
\[
\Vert M_f\Vert \geq \Vert M_f \Vert_{\ess} \geq \frac{1}{2} \Vert M_f \Vert.
\]
In particular $M_f :  H_p(\mathbb{D}^\infty_2) \to  H_p(\mathbb{D}^\infty_2)$ is compact if and only if $f=0$. The same equalities hold if we replace $\mathbb{D}^\infty_2$ by $\mathbb{D}^N$, with $N \in \mathbb{N}$.
\end{proposition}
\begin{proof}
The inequality $\Vert M_f \Vert_{\ess}  \leq \Vert M_f \Vert = \Vert f\Vert_{H_{\infty}(\mathbb{D}^N_2)}$ is already known for every  $N\in \mathbb{N}\cup\{\infty\}$. It is only left, then, to see that
\begin{equation} \label{cilindro}
\Vert M_f \Vert \leq \Vert M_f \Vert_{\ess} \,.
\end{equation}
We begin with the case $N \in \mathbb{N}$. Assume in first place that $p>1$, and take a sequence  $(z^{(n)})_n \subseteq \mathbb{D}^N$, with $\Vert z^{(n)} \Vert_\infty \to 1$, such that $\vert f(z^{(n)}) \vert \to \Vert f \Vert_{H_{\infty}(\mathbb{D}^N)}$. Consider now the function given by
\[
h_{z^{(n)}}(u) = \left( \prod\limits_{j=1}^N \frac{1- \vert z^{(n)}_j\vert^2}{(1- \overline{z^{(n)}_j}u_j)^2}\right)^{1/p},
\] 
for $u \in \mathbb{D}^{N}$. Now, by the Cole-Gamelin inequality \eqref{eq: Cole-Gamelin}
\[
\vert f(z^{(n)})\vert = \vert f(z^{(n)}) \cdot h_{z^{(n)}}(z^{(n)}) \vert \left( \prod\limits_{j=1}^N \frac{1}{1-\vert z^{(n)}_j \vert^2} \right)^{-1/p} \leq \Vert f \cdot h_{z^{(n)}} \Vert_{H_p(\mathbb{D}_2^N)} \leq \Vert f \Vert_{H_\infty(\mathbb{D}^N_2)},
\]
and then $\Vert f \cdot h_{z^{(n)}} \Vert_{H_p(\mathbb{D}^N_2)} \to \Vert f \Vert_{H_{\infty}(\mathbb{D}^N_2)}$. \\
Observe that $\Vert h_{z^{(n)}} \Vert_{H_p(\mathbb{D}^N)} =1$ and that $ h_{z^{(n)}}(u) \to 0$ as $n\to \infty$ for every $u\in \mathbb{D}^N$. Then Lemma~\ref{maciel} $h_{z^{(n)}}$ tends to zero in the weak-star topology and then, since  $H_p(\mathbb{D}^N_2)$ is reflexive (recall that $1<p<\infty$), also in the weak topology. So, if $K$ is a compact operator on $H_p(\mathbb{D}^N_2)$ then $K(h_{z^{(n)}}) \to 0$ and therefore
\begin{multline*}
\Vert M_f - K \Vert  \geq \limsup\limits_{n \to \infty} \Vert f\cdot h_{z^{(n)}} - K(h_{z^{(n)}}) \Vert_{H_p(\mathbb{D}^N_2)} \\
 \geq \limsup\limits_{n\to \infty} \Vert f\cdot h_{z^{(n)}} \Vert_{H_p(\mathbb{D}^N_2)} -\Vert K(h_{z^{(n)}}) \Vert_{H_p(\mathbb{D}^N_2)} =\Vert f\Vert_{H_{\infty}(\mathbb{D}^N_2)}.
\end{multline*}
Thus, $\Vert M_f - K\Vert \geq \Vert f \Vert_{H_{\infty}(\mathbb{D}^N_2)}$ for each compact operator $K$ and hence $\Vert M_f \Vert_{\ess} \geq \Vert M_f\Vert$ as we wanted to see.\\
The proof of  the case $p=1$  follows some ideas of Demazeux in \cite[Theorem~2.2]{demazeux2011essential}. First of all, recall that the $N$-dimensional F\'ejer's Kernel is defined as
\[
K_n^N (u)=\sum\limits_{\vert \alpha_1\vert, \cdots \vert \alpha_N\vert \leq N} \prod\limits_{j=1}^{N} \left(1-\frac{\vert \alpha_j\vert}{n+1}\right) u^{\alpha}\,,
\]
for $u \in \mathbb{D}^N_2$. With this, the $n$-th F\'ejer polynomial with $N$ variables of a function $g\in H_p(\mathbb{D}^N_2)$ is obtained by convoluting $g$ with the $N-$dimensional F\'ejer's Kernel, in other words
\begin{equation} \label{fejerpol}
\sigma_n^N g (u) = \frac{1}{(n+1)^N} \sum\limits_{l_1,\cdots, l_N=1}^{n} \sum\limits_{\vert\alpha_j\vert\leq l_j} \hat{g}(\alpha) u^{\alpha}.
\end{equation}
It is well known (see e.g. \cite[Lemmas~5.21 and~5.23]{defant2018Dirichlet}) that $\sigma_n^N : H_1(\mathbb{D}^N_2) \to H_1(\mathbb{D}^N_2)$ is a contraction and $\sigma_n^N g \to g$ on $H_1(\mathbb{D}^N_2)$ when $n\to \infty$ for all $g\in H_1(\mathbb{D}^N_2)$. Let us see how $R_n^N = I - \sigma_n^N$, gives a first lower bound for the essential norm.\\
Let $K: H_1(\mathbb{D}^N_2) \to H_1(\mathbb{D}^N_2)$ be a compact operator, since $\Vert \sigma_n^N \Vert \leq 1$ then $\Vert R_n^N \Vert \leq 2$ and hence
\[
\Vert M_f - K \Vert \geq \frac{1}{2} \Vert R_n^N \circ (M_f -K) \Vert \geq \frac{1}{2} \Vert R_n^N \circ M_f \Vert - \frac{1}{2} \Vert R_n^N \circ K \Vert.
\]
On the other side, since $R_n^N \to 0$ pointwise,  $R_n^N$ tends to zero uniformly on compact sets of $H_1(\mathbb{D}^N)$. In particular on the compact set $\overline{K(B_{H_1(\mathbb{D}^N)})}$, and therefore $\Vert R_n^N \circ K \Vert \to 0$. We conclude then that $\Vert M_f \Vert_{\ess} \geq \frac{1}{2} \limsup\limits_{n\to\infty} \Vert R_n^N\circ M_f \Vert$.\\
Our aim now is to obtain a lower bound for the right-hand-side of the inequality. To get this, we are going to see that 
\begin{equation} \label{agus}
\Vert \sigma^N_n \circ M_f(h_z) \Vert_{H_1(\mathbb{D}^N)} \to 0 \; \text{when} \; \Vert z \Vert_\infty \to 1,
\end{equation}
where $h_z$ is again defined, for each fixed $z \in \mathbb{D}^{N}$, by 
\[
h_z(u) = \prod\limits_{j=1}^N \frac{1- \vert z_j\vert^2}{(1- \overline{z}_ju_j)^2}.
\]
To see this, let us consider first, for each $z \in \mathbb{D}^{N}$, the function $g_z (u) = \prod\limits_{j=1}^N \frac{1}{(1-\bar{z_j} u_{j})^{2}}$.
This is clearly holomorphic and, hence, has a development a as Taylor series
\[
g_{z}(u) = \sum_{\alpha \in \mathbb{N}_{0}^{N}} c_{\alpha}(g_{z}) u^{\alpha}
\]
for $u \in \mathbb{D}^{N}$. Our first step is to see that the Taylor coefficients up to a fixed degree are bounded uniformly on $z$. Recall that $c_{\alpha}(g_{z})  = \frac{1}{\alpha !} \frac{\partial^{\alpha} g(0)}{\partial u^{\alpha}}$ and, since 
\[
\frac{\partial^{\alpha}g_z(u)}{\partial u^{\alpha}} = \prod\limits_{j=1}^{N} \frac{(\alpha_j + 1)!}{(1- \overline{z_j}u_j)^{2+\alpha_j}} (\overline{z_j})^{\alpha_j},
\]
we have 
\[
c_{\alpha}(g_{z})  
= \frac{1}{\alpha !}\frac{\partial^{\alpha}g_z(0)}{\partial u^{\alpha}} 
=  \frac{1}{\alpha !} \prod\limits_{j=1}^{N} (\alpha_j + 1)!(\overline{z_j})^{\alpha_j}
= \left( \prod\limits_{j=1}^{N} (\alpha_j + 1) \right) \overline{z}^{\alpha} \,.
\]
Thus $\vert c_{\alpha} (g_{z}) \vert \leq (M+1)^{N}$ whenever  $\vert \alpha \vert \leq M$. \\
On the other hand, for each $\alpha \in \mathbb{N}_{0}^{N}$ (note that $h_{z}(u) = g_{z}(u) \prod_{j=1}^{N} (1- \vert z_{j}\vert)$ for every $u$) we have
\[
c_{\alpha} (f\cdot h_z) = \left( \prod\limits_{j=1}^N (1- \vert z_j \vert^2) \right) \sum\limits_{\beta + \gamma =\alpha} \hat{f}(\beta) \hat{g}_z(\gamma) \,.
\]
Taking all these into account we finally have (recall \eqref{fejerpol}), for each fixed $n \in \mathbb{N}$
\begin{align*}
\Vert \sigma_n^N & \circ M_f (h_z) \Vert_{H_1(\mathbb{D}^N)}  \\
& \leq \left( \prod\limits_{j=1}^N 1- \vert z_j \vert^2 \right) \frac{1}{(n+1)^N} \sum\limits_{l_1,\cdots, l_N=1}^{N} \sum\limits_{\vert\alpha_j\vert\leq l_j} \vert \sum\limits_{\beta + \gamma =\alpha} \hat{f}(\beta) \hat{g}_z(\gamma) \vert \Vert u^{\alpha}\Vert_{H_1(\mathbb{D}^N)} \\
&\leq \left( \prod\limits_{j=1}^N 1- \vert z_j \vert^2 \right) \frac{1}{(n+1)^N} \sum\limits_{l_1,\cdots, l_N=1}^{N}\sum\limits_{\vert\alpha_j\vert\leq l_j} \sum\limits_{\beta + \gamma =\alpha} \Vert f \Vert_{H_{\infty}(\mathbb{D}^N)} (N+1)^{N} \,,
\end{align*}
which immediately yields \eqref{agus}. Once we have this we can easily conclude the argument. For each $n\in \mathbb{N}$ we have
\begin{multline*}
\Vert R_n^N \circ M_f \Vert = \Vert M_f - \sigma_n^N \circ M_f \Vert 
\geq \Vert M_f (h_z) - \sigma_n^N \circ M_f (h_z) \Vert_{H_1(\mathbb{D}^N)} \\
\geq \Vert M_f (h_z) \Vert_{H_1(\mathbb{D}^N_2)} - \Vert \sigma_n^N \circ M_f (h_z) \Vert_{H_1(\mathbb{D}^N)},
\end{multline*}
and since the last term tends to zero if $\Vert z\Vert_{\infty} \to 1$, then 
\[
\Vert R_n^N \circ M_f \Vert \geq \limsup\limits_{\Vert z\Vert \to 1} \Vert M_f (h_{z})\Vert_{H_1(\mathbb{D}^N)} \geq \Vert f\Vert_{H_{\infty}(\mathbb{D}^N)} \,,
\]
which finally gives
\[
\Vert M_f \Vert_{\ess} \geq \frac{1}{2} \Vert f\Vert_{H_{\infty}(\mathbb{D}^N_2)} = \frac{1}{2} \Vert M_f \Vert\,,
\]
as we wanted.\\
To complete the proof we consider the case $N=\infty$. So, what we have to see is that
\begin{equation} \label{farola}
\Vert M_f \Vert \geq \Vert M_f \Vert_{\ess} \geq C \Vert M_f \Vert \,,
\end{equation}
where $C=1$ if $p>1$ and $C=1/2$ if $p=1$. Let $K: H_p(\mathbb{D}^\infty_2) \to H_p(\mathbb{D}^\infty_2)$ be a compact operator, and consider for each $N \in \mathbb{N}$ the continuous operators $\mathcal{I}_N : H_p (\mathbb{D}^N) \to H_p(\mathbb{D}^\infty_2)$ given by the inclusion and $\mathcal{J}_N : H_p(\mathbb{D}^\infty_2) \to H_p ( \mathbb{D}^N)$ defined by $\mathcal{J}(g)(u)= g(u_1,\cdots, u_N, 0) = g_N(u)$ then $K_N =\mathcal{J}_{N} \circ K \circ \mathcal{I}_{N}: H_p(\mathbb{D}^N) \to H_p(\mathbb{D}^N)$ is compact. On the other side we have that $\mathcal{J}_N \circ M_f \circ \mathcal{I}_{N} (g) = f_n\cdot g = M_{f_N} (g)$ for every $g$, furthermore given any operator $T:H_p(\mathbb{D}^\infty_2) \to H_p(\mathbb{D}^\infty_2)$ and defining $T_N$ as before we have that
\begin{align*}
\Vert T \Vert =\sup\limits_{ \Vert g\Vert_{H_p(\mathbb{D}^\infty_2)}\leq 1} \Vert T(g) \Vert_{H_p(\mathbb{D}^\infty_2)} 
& \geq  \sup\limits_{ \Vert g\Vert_{H_p(\mathbb{D}^N)}\leq 1} \Vert T(g) \Vert_{H_p(\mathbb{D}^\infty_2)} \\
& \geq \sup\limits_{ \Vert g\Vert_{H_p(\mathbb{D}^N)}\leq 1} \Vert T_M(g) \Vert_{H_p(\mathbb{D}^N_2)} =\Vert T_N \Vert,
\end{align*}
and therefore
\[
\Vert M_f - K \Vert \geq \Vert M_{f_N} -K_N \Vert 
\geq \Vert M_{f_N} \Vert_{\ess} \geq C \Vert f_N \Vert_{H_{\infty}(\mathbb{D}^N_2)}\,.
\]
Since $\Vert f_{N} \Vert_{H_{\infty}(\mathbb{D}^N_2)} \to \Vert f \Vert_{H_{\infty}(\mathbb{D}^\infty_2)}$ when $N \to \infty$ we have \eqref{farola}, and this completes the proof.
\end{proof}

\noindent We can now prove Theorem~\ref{saja}.

\begin{proof}[Proof of Theorem~\ref{saja}]
Since for every $1\leq p < \infty$ the Bohr lift $\mathcal{L}_{\mathbb{D}^N_2} : \mathcal{H}_p^{(N)} \to H_p(\mathbb{D}^N_2)$ and the Bohr transform $\mathcal{B}_{\mathbb{D}^N_2} : H_p(\mathbb{D}^N_2) \to \mathcal{H}_p^{(N)}$ are isometries, then an operator $K : \mathcal{H}_p^{(N)} \to \mathcal{H}_p^{(N)}$ is compact if and only if $K_h = \mathcal{L}_{\mathbb{D}^N_2} \circ K \circ \mathcal{B}_{\mathbb{D}^N_2} : H_p(\mathbb{D}^N_2) \to H_p(\mathbb{D}^N_2)$ is a compact operator.
On the other side $f= \mathcal{L}_{\mathbb{D}^N_2}(D)$ hence $M_f = \mathcal{L}_{\mathbb{D}^N_2} \circ M_D \circ \mathcal{B}_{\mathbb{D}^N_2}$ and therefore
\[
\Vert M_D - K \Vert = \Vert \mathcal{L}_{\mathbb{D}^N_2}^{-1} \circ ( M_f - K_h ) \circ \mathcal{L}_{\mathbb{D}^N_2} \Vert = \Vert M_f - K_h \Vert \geq C \Vert f \Vert_{H_\infty(\mathbb{D}^N_2)} = C \Vert D \Vert_{\mathcal{H}_\infty^{(N)}},
\]
where $C=1$ if $p>1$ and $C= 1/2$ if $p=1$.
Since this holds for every compact operator $K$ then we have the inequality that we wanted. The upper bound is clear by the definition of essential norm.

On the other hand, if $p=1$ and $N \in \mathbb{N} \cup\{\infty\}$. Let $1 < q < \infty$ an consider $M_D^q : \mathcal{H}_q^{(N)} \to \mathcal{H}_1^{(N)}$ the restriction. If $K: \mathcal{H}_1^{(N)} \to \mathcal{H}_1^{(N)}$ is compact then its restriction $K^q : \mathcal{H}_q^{(N)} \to \mathcal{H}_1^{(N)}$ is also compact and then 
\begin{align*}
\Vert M_D - K \Vert_{\mathcal{H}_1^{(N)} \to \mathcal{H}_1^{(N)}} &= \sup\limits_{\Vert E \Vert_{\mathcal{H}_1^{(N)}} \leq 1} \Vert M_D(E) - K(E) \Vert_{\mathcal{H}_1^{(N)}} \\
&\geq \sup\limits_{\Vert E \Vert_{\mathcal{H}_q^{(N)} \leq 1}} \Vert M_D(E) - K(E) \Vert_{\mathcal{H}_1^{(N)}} \\
&= \Vert M_D^q - K^q \Vert_{\mathcal{H}_q^{(N)} \to \mathcal{H}_1^{(N)}} \geq \Vert M_D^q \Vert_{\ess} \geq \Vert D \Vert_{\mathcal{H}_1^{(N)}}.
\end{align*}
Finally, the case $p=\infty$ was proved in  \cite[Corollary~2,4]{lefevre2009essential}.
\end{proof}

\section{Spectrum of Multiplication operators}

In this section, we provide a characterization of the spectrum of the multiplication operator $M_D$, with respect to the image of its associated Dirichlet series in some specific half-planes. Let us first recall some definitions of the spectrum of an operator. We say that $\lambda$ belongs to the spectrum of $M_D$, that we note $\sigma(M_D)$, if the operator $M_D - \lambda I : \mathcal{H}_p \to \mathcal{H}_p$ is not invertible. Now, a number $\lambda$ can be in the spectrum for different reasons and according to these we can group them into the following subsets:
\begin{itemize}
\item If $M_D - \lambda I$ is not injective then $\lambda \in \sigma_p(M_D)$, the point spectrum.
\item If $M_D-\lambda I$ is injective and the $Ran(A-\lambda I)$ is dense (but not closed) in $\mathcal{H}_p$ then $\lambda \in \sigma_c(M_D)$, the continuous spectrum of $M_D$.
\item If $M_D-\lambda I$ is injective and its range has codimension greater than or equal to 1 then $\lambda$ belongs to $\sigma_r(M_D)$, the radial spectrum.
\end{itemize}

We are also interested in the approximate spectrum, noted by $\sigma_{ap}(M_D)$, given by those values $\lambda \in \sigma(M_D)$ for which there exist a unit sequence $(E_n)_n \subseteq \mathcal{H}_p$ such that $\Vert M_D(E_n) - \lambda E_n \Vert_{\mathcal{H}_p} \to 0$.

Vukoti\'c, in \cite[Theorem~7]{vukotic2003analytic}, proved that the spectrum of a Multiplication operator, induced by function $f$ in the one dimensional disk, coincides with $\overline{f(\mathbb{D})}$. In the case of the continuous spectrum, the description is given from the outer functions in $H_\infty(\mathbb{D})$. 
The notion of outer function can be extended to higher dimensions. If $N\in \mathbb{N}\cup\{\infty\}$, a function $f\in H_p(\mathbb{D}^N_2)$ is said to be outer if it satisfies
\[
\log\vert f(0) \vert = \int\limits_{\mathbb{T}^N} \log\vert F(\omega)\vert \mathrm{d}\omega,
\]
with $f\sim F$. A closed subspace $S$ of $H_p(\mathbb{D}^N_2)$ is said to be invariant, if for every $g\in S$ it is verified that $z_i \cdot g \in S$ for every monomial. Finally, a function $f$ is said to be cyclic, if the invariant subspace generated by $f$ is exactly $H_p(\mathbb{D}^N_2)$. The mentioned characterization comes from the generalized Beurling's Theorem, which affirms that $f$ is a cyclic vector if and only if $f$ is an outer function. In several variables, there exist outer functions which fail to be cyclic (see   \cite[Theorem~4.4.8]{rudin1969function}).
We give now the aforementioned characterization of the spectrum of a multiplication operator.

\begin{theorem} \label{espectro}
Given $1\leq p <\infty$ and  $D\in \mathcal{H}_{\infty}$ a non-zero Dirichlet series with associated multiplication operator $M_D : \mathcal{H}_p \to \mathcal{H}_p$. Then
\begin{enumerate}
\item \label{espectro1}  $M_D$ is onto if and only if there is some $c>0$ such that  $\vert D (s) \vert \geq c$ for every $s \in \mathbb{C}_{0}$.
\item \label{espectro2}  $\sigma(M_D)=\overline{D(\mathbb{C}_0)}$.
\item \label{espectro3} If $D$ is not constant then $\sigma_c(M_D) \subseteq \overline{D(\mathbb{C}_0)} \setminus D(\mathbb{C}_{1/2})$. Even more, if $\lambda \in \sigma_c(M_D)$ then $f - \lambda = \mathcal{L}_{\mathbb{D}^\infty_2}(D) - \lambda$ is an outer function in $H_{\infty}(\mathbb{D}^\infty_2)$.
\end{enumerate}
The same holds if we replace in each case $\mathcal{H}$ by $\mathcal{H}^{(N)}$ (with $N \in \mathbb{N}$).
\end{theorem}

\begin{proof}
\ref{espectro1} Because of the injectivity of $M_D$, and the Closed Graph Theorem, the mapping $M_D$ is surjective if and only if $M_D$ is invertible and this happens if and only if $M_{D^{-1}}$ is well defined and continuous, but then $D^{-1} \in \mathcal{H}_{\infty}$ and \cite[Theorem~6.2.1]{queffelec2013diophantine} gives the conclusion.

\ref{espectro2} Note that $M_D - \lambda I = M_{D-\lambda}$; this and the previous result give that $\lambda \not\in \sigma( M_D)$ if and only if $\vert D(s) - \lambda \vert > \varepsilon$ for some $\varepsilon >0$ and all $s\in \mathbb{C}_0$, and this happens if and only if $\lambda \not\in \overline{D(\mathbb{C}_0)}$.

\ref{espectro3}  Let us suppose that the range of $M_D - \lambda = M_{D-\lambda}$ is dense.  Since polynomials are dense in $\mathcal H_p$ and $M_{D-\lambda}$ is continuous then $A:=\{ (D-\lambda)\cdot P : P \; \text{Dirichlet polynomial} \}$ is dense in the range of $M_{D-\lambda}$. 
By the continuity of the  evaluation at $s_0 \in \mathbb{C}_{1/2}$,  the set of Dirichlet series that vanish in a fixed $s_0$, which we denote by $B(s_0)$, is a proper closed set (because $1 \not\in B(s_0)$). Therefore, if $D-\lambda \in B(s_0)$ then $A\subseteq B(s_0)$, but hence $A$ cannot be dense in $\mathcal{H}_p$. So we have that if $\lambda \in \sigma_c(M_D)$ then $D(s) - \lambda \not= 0$ for every $s\in \mathbb{C}_{1/2}$ and therefore $\lambda \in \overline{D(\mathbb{C}_0)} - D(\mathbb{C}_{1/2})$.

Finally, since $\sigma_c(M_D) =  \sigma_c(M_f)$ then $\lambda \in \sigma_c(M_D)$ if and only if  $M_{f-\lambda}(H_p(\mathbb{D}^\infty_2))$ is dense in $H_p(\mathbb{D}^\infty_2)$. Consider $S(f-\lambda)$ the smallest closed subspace of $H_p(\mathbb{D}^\infty_2)$ such that $z_i\cdot (f-\lambda) \in S(f-\lambda)$ for every $i \in \mathbb{N}$.
Take  $\lambda \in \sigma_c(M_f)$ and note that
\[
\{ (f-\lambda)\cdot P : P \; \text{polynomial} \} \subseteq S(f-\lambda) \subseteq H_p(\mathbb{D}^\infty_2) \,.
\]
Since the polynomials are dense in $H_p(\mathbb{D}^\infty_2)$, and $S(f - \lambda)$ is closed, we obtain that $S(f-\lambda) = H_p(\mathbb{D}^\infty_2)$.
Then $f-\lambda$ is a cyclic vector in $H_{\infty}(\mathbb{D}^\infty_2)$ and therefore the function $f-\lambda \in H_{\infty}(\mathbb{D}^\infty_2)$ is an outer function (see \cite[Corollary~5.5]{guo2022dirichlet}). 
\end{proof}

Note that, in the hypothesis of the previous Proposition, if $D$ is non-constant, then $\sigma_p(M_D)$ is empty and therefore, $\sigma_r(M_D) = \sigma(M_D) \setminus \sigma_c(M_D)$. As a consequence, $\sigma_r(M_D)$ must contain the set $D(\mathbb{C}_{1/2})$.

Note that a value $\lambda$ belongs to the approximate spectrum of a multiplication operator $M_D$ if and only if $M_{D} - \lambda I = M_{D-\lambda}$ is  not bounded from below. If $D$ is not constant and equal to $\lambda$ then, $M_{D-\lambda}$ is injective. Therefore, being bounded from below is equivalent to having closed ranged. Thus, we need to understand when does this operator have closed range. We therefore devote some lines to discuss this property. 

The range of the multiplication operators behaves very differently depending on whether or not it is an endomorphism. We see now that if $p\not= q$ then  multiplication operators never have closed range.

\begin{proposition} \label{prop: rango no cerrado}
Given $1\leq q < p \leq \infty$ and $D\in \mathcal{H}_t$, with $t=pq/(p-q)$ if $p< \infty$ and $t= q$ if $p= \infty$, then $M_D : \mathcal{H}_p \to \mathcal{H}_q$ does not have a closed range. The same holds if we replace $\mathcal{H}$ by $\mathcal{H}^{(N)}$ (with $N \in \mathbb{N}$).
\end{proposition}

\begin{proof}
Since $M_D : \mathcal{H}_p \to \mathcal{H}_q$ is injective, the range of $M_D$ is closed if and only if there exists $C>0$ such that $C \Vert E \Vert_{\mathcal{H}_p} \leq \Vert D\cdot E \Vert_{\mathcal{H}_q}$ for every $E\in \mathcal{H}_p$. Suppose that this is the case and choose some Dirichlet polynomial $P\in \mathcal{H}_t$  such that $\Vert D - P \Vert_{\mathcal{H}_t} < \frac{C}{2}$. Given  $E\in \mathcal{H}_p$ we have
\begin{multline*}
\Vert P \cdot E \Vert_{\mathcal{H}_q} 
= \Vert D\cdot E - (D-P) \cdot E \Vert_{\mathcal{H}_q} 
\geq \Vert D \cdot E \Vert_{\mathcal{H}_q} - \Vert ( D - P ) \cdot E \Vert_{\mathcal{H}_q} \\
\geq C \Vert E \Vert_{\mathcal{H}_p} - \Vert D - P \Vert_{\mathcal{H}_t} \Vert E \Vert_{\mathcal{H}_p} 
\geq \frac{C}{2} \Vert E \Vert_{\mathcal{H}_p}.
\end{multline*}
Then $M_P : \mathcal{H}_p \to \mathcal{H}_q$ has closed range. Let now $(Q_n)_n$ be a sequence of polynomials converging in $\mathcal{H}_q$ but not in $\mathcal{H}_p$, then
\[
C\Vert Q_n - Q_m \Vert_{\mathcal{H}_p} \leq \Vert P \cdot (Q_n -Q_m) \Vert_{\mathcal{H}_q} \leq \Vert P \Vert_{\mathcal{H}_\infty} \Vert Q_n - Q_m \Vert_{\mathcal{H}_q},
\]
which is a contradiction.
\end{proof}

As we mentioned before, the behaviour of the range is very different when the operator is an endomorphism, that is, when $p=q$. Recently, in \cite[Theorem~4.4]{antezana2022splitting}, Antenaza, Carando and Scotti have established a series of equivalences for certain Riesz systems in  $L_2(0,1)$. Within the proof of this result, they also characterized those Dirichlet series $D\in \mathcal{H}_\infty$, for which their associated multiplication operator $M_D: \mathcal{H}_p \to \mathcal{H}_p$ has  closed range. The proof also works for $\mathcal H_p$. In our aim to be as clear and complete as possible, we develop below the arguments giving all the necessary definitions.

A character is a function $\gamma: \mathbb{N} \to \mathbb{C}$ that satisfies
\begin{itemize}
\item $\gamma (m n) = \gamma(m) \gamma (n)$ for all $m,n \in \mathbb{N}$,
\item $\vert \gamma (n) \vert =1$ for all $n \in \mathbb{N}$.
\end{itemize}
The set of all characters is denoted by $\Xi$. 
Given a Dirichlet series $D= \sum a_n n^{-s}$, each character $\gamma \in \Xi$ defines a new Dirichlet series by 
\begin{equation}\label{caracter}
D^\gamma (s) =\sum a_n \gamma(n) n^{-s}.
\end{equation}
Each character $\gamma \in\Xi$ can be identified with an element $\omega \in \mathbb{T}^{\infty}$, taking $\omega = (\gamma ( \mathfrak{p}_1) , \gamma(\mathfrak{p}_2), \cdots )$, and then we can rewrite \eqref{caracter} as
\[
D^\omega (s) =\sum a_n \omega(n)^{\alpha(n)} n^{-s},
\]
being $\alpha(n)$ such that $n= \mathfrak{p}^{\alpha(n)}$.

Note that if $\mathcal{L}_{\mathbb{T}^\infty}(D)(u) = F(u) \in H_\infty(\mathbb{T}^\infty),$ then by comparing coefficients we have that $\mathcal{L}_{\mathbb{T}^\infty}(D^\omega)(u) = F(\omega\cdot u) \in H_\infty(\mathbb{T}^\infty)$. By \cite[Lemma~11.22]{defant2018Dirichlet}, for all $\omega \in \mathbb{T}^\infty$ the limit
\[
\lim\limits_{\sigma\to 0} D^\omega(\sigma + it), \; \text{exists for almost all} \;  t\in \mathbb{R}.
\]
Using \cite[Theorem~2]{saksman2009integral}, we can choose a representative $\tilde{F}\in H_\infty(\mathbb{T}^\infty)$ of $F$ which satisfies
\begin{equation*}
    \tilde{F}(\omega)=
    \left\{
    \begin{aligned}
    &\lim\limits_{\sigma\to 0^+} D^\omega(\sigma) \; &\text{if the limit exists}; \\
     &0 \; &\text{otherwise}.
    \end{aligned}
    \right.
\end{equation*}
To see this, consider 
\[
A:=\{ \omega \in \mathbb{T}^\infty : \lim\limits_{\sigma\to 0} D^\omega(\sigma) \; \text{exists}. \},
\]
and let us see that $\vert A \vert =1$. To that, take $T_t: \mathbb{T}^\infty \to \mathbb{T}^\infty$ the Kronecker flow defined by $T_t(\omega)=(\mathfrak{p}^{-it} \omega),$ and notice that $T_t(\omega)\in A$ if and only if $\lim\limits_{\sigma\to 0} D^{T_t(\omega)}(\sigma)$ exists. Since
\[
D^{T_t(\omega)}(\sigma)= \sum a_n (\mathfrak{p}^{-it} \omega)^{\alpha(n)} n^{-\sigma}= \sum a_n \omega^{\alpha(n)} n^{-(\sigma+it)} = D^{\omega}(\sigma+it),
\]
then for all $\omega\in \mathbb{T}^\infty$ we have that $T_t(\omega) \in A$ for almost all $t\in \mathbb{R}.$
Finally, since $\chi_A \in L^1(\mathbb{T}^\infty),$ applying the Birkhoff Theorem for the Kronecker flow \cite[Theorem 2.2.5]{queffelec2013diophantine}, for $\omega_0 = (1,1,1,\dots)$ we have
\[
\vert A \vert = \int\limits_{\mathbb{T}^\infty} \chi_A(\omega) \mathrm{d}\omega = \lim\limits_{R\to \infty} \frac{1}{2R} \int\limits_{-R}^{R} \chi_A (T_t(\omega_0)) \mathrm{d}t = 1.
\]

Then $\tilde{F} \in H_\infty (\mathbb{T}^\infty),$ and to see that $\tilde{F}$ is a representative of $F$ it is enough to compare their Fourier coefficients (see again \cite[Theorem~2]{saksman2009integral}). From now to the end $F$ is always $\tilde{F}$.\\
Fixing the notation 
\[
D^\omega(it_0)= \lim\limits_{\sigma\to 0} D^\omega(\sigma +it),
\]
then taking $t_0= 0,$ we get 
\[
F(\omega) = D^\omega(0)
\]
for almost all $\omega \in \mathbb{T}^\infty$.
Moreover, given $t_0 \in \mathbb{R}$ we have
\begin{equation}\label{igualdad}
D^\omega(it_0) = \lim\limits_{\sigma\to 0^+} D^\omega(\sigma + it_0) = \lim\limits_{\sigma\to 0^+} D^{T_{t_0}(\omega)} (\sigma) = F(T_{t_0}(\omega)).
\end{equation}
From this identity one has the following.

\begin{proposition}\label{acotacion}
The followings conditions are equivalent.
\begin{enumerate}
    \item\label{acotacion1} There exists $\tilde{t}_0$ such that $\vert D^{\omega} (i\tilde{t}_0) \vert \geq \varepsilon$ for almost all $\omega \in \mathbb{T}^\infty$.
    \item\label{acotacion2} For all $t_0$ there exists $B_{t_0} \subset \mathbb{T}^\infty$ with total measure such that $\vert D^\omega(it_0) \vert \geq \varepsilon$ for all $\omega \in B_{t_0}$.
\end{enumerate}
\end{proposition}
\begin{proof}
If~\ref{acotacion1}, holds take $t_0$ and consider 
\[
B_{t_0} = \{\mathfrak{p}^{-i(-t_0+\tilde{t}_0)}\cdot \omega : \; \omega\in B_{\tilde{t}_0} \},
\]
which is clearly a total measure set. Take $\omega{'} \in B_{t_0}$ and choose $\omega \in B_{\tilde{t}_0}$ such that $\omega{'} = \mathfrak{p}^{-i(-t_0+\tilde{t}_0)}\cdot \omega$, then by \eqref{igualdad} we have that 
\[
\vert D^{\omega{'}} (it_0) \vert = \vert F(T_{\tilde{t}_0}(\omega)) \vert \geq \varepsilon\,,
\]
and this gives~\ref{acotacion2}. The converse implications holds trivially.
\end{proof}

We now give an $\mathcal H_p$-version of \cite[Theorem~4.4.]{antezana2022splitting}.

\begin{theorem}\label{ACS}
Let $1\leq p < \infty$, and $D \in \mathcal{H}_\infty$. Then the following statements are equivalent.
\begin{enumerate}
    \item\label{ACS1} There exists $m>0$ such that $\vert F(\omega) \vert \geq M$ for almost all $\omega\in \mathbb{T}^\infty$;
    \item\label{ACS2} The operator $M_D : \mathcal{H}_p \to \mathcal{H}_p$ has closed range;
    \item\label{ACS3} There exists $m>0$ such that for almost all $(\gamma, t) \in \Xi \times \mathbb{R}$ we have 
    \[
    \vert D^\gamma(it) \vert\geq m.
    \]
\end{enumerate}
Even more, in that case,
\begin{multline*}
    \inf\left\{\Vert M_D(E) \Vert_{\mathcal{H}_p} : E\in \mathcal{H}_p, \Vert E \Vert_{\mathcal{H}_p}=1 \right\} \\ = \essinf \left\{ \vert F(\omega) \vert : \omega \in \mathbb{T}^\infty \right\}
    = \essinf \left\{ \vert D^\gamma(it) \vert : (\gamma,t)\in \Xi \times \mathbb{R} \right\}.
\end{multline*}
\end{theorem}
\begin{proof}
\ref{ACS1} $\Rightarrow$~\ref{ACS2} $M_D$ has closed range if and only if the rage of $M_F$ is closed. Because of the injectivity of $M_F$ we have, by Open Mapping Theorem, that $M_F$ has closed range if and only if there exists a positive constant $m>0$ such that
\[
\Vert M_F(G) \Vert_{H_p(\mathbb{T}^\infty)} \geq m \Vert G \Vert_{H_p(\mathbb{T}^\infty)},
\]
for every $G\in H_p(\mathbb{T}^\infty)$. If $\vert F(\omega)\vert \geq m$ a.e. $\omega \in \mathbb{T}^\infty$, then for $G \in H_p(\mathbb{T}^\infty)$ we have that
\[
\Vert M_F (G) \Vert_{H_p(\mathbb{T}^\infty)} = \Vert F\cdot G \Vert_{H_p(\mathbb{T}^\infty)} =\left(\int\limits_{\mathbb{T}^\infty} \vert FG(\omega)\vert^p \mathrm{d} \omega\right)^{1/p} \geq m \Vert G\Vert_{H_p(\mathbb{T}^\infty)}.
\]

\ref{ACS2} $\Rightarrow$~\ref{ACS1} Let $m>0$ be such that $\Vert M_F(G)\Vert_{H_p(\mathbb{T}^\infty)} \geq m \Vert G \Vert_{H_p(\mathbb{T}^\infty)}$ for all $G\in H_p(\mathbb{T}^\infty)$. Let us consider 
\[
A=\{ \omega\in \mathbb{T}^\infty : \vert F(\omega) \vert <m\}.
\]
Since $\chi_A \in L^p(\mathbb{T}^\infty)$, by the density of the trigonometric polynomials in $L^p(\mathbb{T}^\infty)$ (see \cite[Proposition~5.5]{defant2018Dirichlet}) there exist a sequence $(P_k)_k$ of degree $n_k$ in $N_k$ variables (in $z$ and $\overline{z}$) such that 
\[
\lim\limits_{k} P_k = \chi_A \; \text{in} \; L^p(\mathbb{T}^\infty).
\]
Therefore 
\begin{align*}
m^p\vert A \vert &= m^p\Vert \chi_A \Vert^p_{L^p(\mathbb{T}^\infty)} = m^p\lim\limits_k \Vert P_k \Vert^p_{L^p(\mathbb{T}^\infty)}\\
&=m^p\lim\limits_k \Vert z_1^{n_k} \cdots z_{N_k}^{n_k} P_k \Vert^p_{L_p(\mathbb{T}^\infty)}\\
&\leq \liminf\limits_k \Vert M_F(z_1^{n_k} \cdots z_{N_k}^{n_k} P_k) \Vert^p_{L_p(\mathbb{T}^\infty)}\\
&= \Vert F\cdot \chi_A \Vert^p_{L^p(\mathbb{T}^\infty)} = \int\limits_{A} \vert F(\omega) \vert^p \mathrm{d}\omega. 
\end{align*}
Since $\vert F(\omega) \vert < m$ for all $\omega \in A$, this implies that $\vert A \vert =0$.

\ref{ACS2} $\Rightarrow$~\ref{ACS3} By the definition of $F$ we have $m \leq \vert F(\omega) \vert = \lim\limits_{\sigma\to 0^+} \vert D^\omega (\sigma) \vert$ for almost all $\omega \in \mathbb{T}^\infty$. Combining this with Remark~\ref{acotacion} we get that the $t-$sections of the set
\[
C= \{ (\omega, t ) \in \mathbb{T}^\infty \times \mathbb{R} : \; \vert D^\omega(it) \vert < \varepsilon \},
\]
have zero measure. As a corollary of Fubini's Theorem we get that $C$ has measure zero. The converse~\ref{ACS3} $\Rightarrow$~\ref{ACS2} also follows from Fubini's Theorem.
The last equality follows from the proven equivalences.
\end{proof}

In the case of polynomials, using the continuity of the polynomials and Kronecker's theorem (see e.g. \cite[Proposition~3.4]{defant2018Dirichlet}), the condition of Theorem~\ref{ACS} is restricted to the image of the polynomial on the line of complex with zero real part. As a consequence, one can extend this characterization to the Dirichlet series belonging to $\mathcal{A}(\mathbb{C}_0)$, that is the closed subspace of $\mathcal{H}_\infty$ given by the Dirichlet series that are uniformly continuous on $\mathbb{C}_0$ (see \cite[Definition~2.1]{aron2017dirichlet}).

\begin{corollary}\label{torres}
Let $1\leq p < \infty$ then

\begin{enumerate}
\item\label{torres1} Let $P\in \mathcal{H}_\infty$ be a Dirichlet polynomial. Then $M_P: \mathcal{H}_p \to \mathcal{H}_p$ has closed range if and only if there exists a constant $m>0$ such that $\vert P(it) \vert \geq m$ for all $t\in \mathbb{R}$. 
\item\label{torres2} Let $D\in \mathcal{A}(\mathbb{C}_0)$, then $M_D: \mathcal{H}_p \to \mathcal{H}_p$ has closed range if and only if there exists a constant $m>0$ such that $\vert D(it) \vert \geq m$ for all $t\in \mathbb{R}$. 
\end{enumerate}
Even more, in each case
\[
\inf \{ \Vert M_D(E) \Vert_{\mathcal{H}_p} : E \in \mathcal{H}_p,\; \Vert E \Vert_{\mathcal{H}_p}=1 \} = \inf \{ \vert D(it) \vert : t\in \mathbb{R} \}.
\]
\end{corollary}

\begin{proof}
\ref{torres1} Let $F= \mathcal{L}_{\mathbb{T}^\infty}(P)$ then, by Theorem~\ref{ACS}, $M_P$ has close range if and only if there exists a constant $m>0$ such that $\vert F (\omega) \vert \geq m$ a.e. $\omega \in \mathbb{T}^\infty$. Since $F(\omega)= \sum a_{\alpha} \omega^\alpha$ is continuous and by Kronecker's theorem
\[
\{(\mathfrak{p}_1^{-it},\cdots, \mathfrak{p}_N^{-it}, \omega) \; : \; t \in \mathbb{R}, \; \omega \in \mathbb{T}^\infty \}
\]
is dense in $\mathbb{T}^\infty$, then $M_P$ has closed range if and only if $\vert F(\mathfrak{p}_1^{-it},\cdots, \mathfrak{p}_N^{-it}, \omega) \vert \geq m$ for every $t \in \mathbb{R}$ and $\omega \in \mathbb{T}^\infty$. The result is concluded from the fact that 
\begin{equation*}
F(\mathfrak{p}_1^{-it},\cdots, \mathfrak{p}_N^{-it}, \omega) = \sum a_{\alpha} \mathfrak{p}_1^{-it\alpha_1}\cdots \mathfrak{p}_N^{-it\alpha_N} = \sum a_n n^{-it} = P(it).
\end{equation*}

\ref{torres2} Since $D$ is uniformly continuous on $\mathbb{C}_0$ then $D$ admits a uniformly continuous extension to the half-plane $\{s\in \mathbb{C} : \re s\geq 0\}.$ By \cite[Theorem~2.3]{aron2017dirichlet}, $D$ is the uniform limit on $\mathbb{C}_0$ of a sequence of Dirichlet polynomials $P_n$. Let $\mathcal{A}(\mathbb{T}^\infty)$ be the closed subspace of $H_\infty(\mathbb{T}^\infty)$ given by the Bohr transform of $\mathcal{A}(\mathbb{C}_0)$. If $\mathcal{L}_{\mathbb{T}^\infty}(D) = F \in \mathcal{A}(\mathbb{T}^\infty)$, since it is the uniform limit of polynomials, then $F$ is continuous. Then, given $t\in \mathbb{R}$ we have that

\begin{align}\label{borde}
\vert F(\mathfrak{p}^{-it}) \vert = \lim\limits_n \vert \mathcal{B}_{\mathbb{T}^\infty} (P_n) (\mathfrak{p}^{-it}) \vert = \lim\limits_n \vert P_n(it) \vert = \vert D(it) \vert.
\end{align}
Let us suppose first that the range of $M_D:\mathcal{H}_p \to \mathcal{H}_p$ is closed and let $m>0$ be such that $\Vert M_D(E) \Vert_{\mathcal{H}_p} \geq \Vert E \Vert_{\mathcal{H}_p}$. Given $\varepsilon >0$ there exists $n_0\in\mathbb{N}$ such that $\Vert D - P_n \Vert_{\mathcal{H}_\infty} <\varepsilon$  for every $n_0 \leq n$. Therefore, if $E \in \mathcal{H}_p$ we have that
\[
\Vert M_{P_n} (E) \Vert_{\mathcal{H}_p} \geq \Vert M_D(E) \Vert_{\mathcal{H}_p} - \Vert M_{D-P_n}(E) \Vert_{\mathcal{H}_p} \geq (m - \varepsilon) \Vert E \Vert_{\mathcal{H}_p}.
\]
Then by~\ref{torres1}, $\vert P_n(it) \vert \geq m-\varepsilon$ for every $n\geq n_0$ and $t\in \mathbb{R}$ and hence, by \eqref{borde}, $\vert D(it)\vert \geq m - \varepsilon$ for every $\varepsilon >0$.

Let us suppose now that there exists a constant $m>0$ such that $\vert D(it) \vert \geq m$ for every $t\in \mathbb{R}$, then from \eqref{borde} we have that $\vert F(\mathfrak{p}^{-it})\vert \geq m$ for all $t\in \mathbb{R}$. Since $F$ is continuous, and by Kronecker's theorem, $\vert F(\omega) \vert \geq m$ for all $\omega \in \mathbb{T}^\infty.$ Therefore, by Theorem~\ref{ACS}, $M_D$ has closed range.
\end{proof}

 For what was said above, in the non trivial case, a value $\lambda$ belongs to the approximate spectrum of $M_D$ if and only if the range of $M_{D-\lambda}$ is not closed. Then, Theorem ~\ref{ACS} and Proposition~\ref{torres} give us a characterization of the approximate spectrum. For this, we need the definition of the essential range of the function $[(\gamma,t) \rightsquigarrow D^{\gamma}(it)]$. That is, 
\[
\Big\{ \lambda \in \mathbb{C} : \text{ for all } \varepsilon>0, \; \mu\big\{(\gamma,t): \vert D^{\gamma}(it) - \lambda  \vert < \varepsilon \big\}>0 \Big\},
\]
where $\mu$ stands for the Haar measure in $\Xi \times\mathbb{R}$.
\begin{theorem}\label{aproximado}
Let $1\leq p < \infty$
\begin{enumerate}
\item\label{aproximado1} If $D\in \mathcal{H}_\infty$, then $\sigma_{ap}(M_D)=essran [(\gamma,t) \rightsquigarrow D^{\gamma}(it)]$.
\item\label{aproximado2} If $D\in \mathcal{A}(\mathbb{C}_0)$, then $\sigma_{ap}(M_D) = \overline{\{  D(it) : t\in\mathbb{R} \}}.$
\end{enumerate}
\end{theorem}

\begin{proof}
\ref{aproximado1} A value  $\lambda$ belongs to $\sigma_{ap}(M_D)$ if and only if the range of $M_{D-\lambda}$ is not closed; and by Theorem~\ref{ACS}, if and only if 
\[
\essinf \{ \vert D^\gamma(it)- \lambda\vert: (\gamma, t) \in \Xi \times\mathbb{R}\}=\essinf \{ \vert (D-\lambda)^\gamma(it)\vert: (\gamma, t) \in \Xi \times\mathbb{R}\} = 0,
\]
but that is equivalent to say that the measure of $\{\vert  D^\gamma (it) - \lambda \vert < \varepsilon : (\gamma, t) \in \Xi\times\mathbb{R}\}$ is bigger than zero for every $\varepsilon >0$. In other words, $\lambda$ belongs to the essential range of $[(\gamma,t) \rightsquigarrow D^{\gamma}(it)]$.

\ref{aproximado2} Following the same arguments used in~\ref{aproximado1} and using Corollary~\ref{torres} we have that $\lambda \in \sigma_{ap}(M_D)$ if and only if $\inf\{ \vert D(it) - \lambda \vert : t\in \mathbb{R} \} = 0,$
if and only if $\lambda \in \overline{\{  D(it) : t\in\mathbb{R} \}}.$
\end{proof}

\noindent 
T.~Fern\'andez Vidal, D.~Galicer\\
Departamento de Matem\'{a}tica,
Facultad de Cs. Exactas y Naturales, Universidad de Buenos Aires and IMAS-CONICET. Ciudad Universitaria, Pabell\'on I (C1428EGA) C.A.B.A., Argentina, tfvidal@dm.uba.ar, dgalicer@dm.uba.ar\\ 

\noindent P.~Sevilla-Peris\\	
Insitut Universitari de Matem\`atica Pura i Aplicada. Universitat Polit\`ecnica de Val\`encia. Cmno Vera s/n 46022, Spain, psevilla@mat.upv.es

\end{document}